\documentclass[12pt,a4paper]{amsart}
\usepackage{amsfonts}
\usepackage{amssymb}
\usepackage{amsmath}
\usepackage{amsthm}
\usepackage{cite}
\usepackage{enumitem}
\usepackage{tikz}

\usepackage{caption}
\usepackage{subfigure}

\usepackage{MnSymbol}

\theoremstyle{plain}
\newtheorem{thm}{Theorem}

\newtheorem{prop}{Proposition}

\newtheorem{bem}{Remark}
\newtheorem*{thm*}{Theorem}
\newtheorem*{prop*}{Proposition}

\newcommand{\vecIII}[3]{
\ensuremath{
\begin{pmatrix}
#1 \\ #2 \\ #3 \\
\end{pmatrix}}}

\newcommand{\matIII}[9]{
\ensuremath{ 
\begin{pmatrix}  
#1 & #2 & #3\\
#4 & #5 & #6\\ 
#7 & #8 & #9\\
\end{pmatrix}}}

\providecommand{\sm}{\setminus}
\providecommand{\N}{\mathbb{N}}
\providecommand{\R}{\mathbb{R}}
\providecommand{\Z}{\mathbb{Z}}
\providecommand{\C}{\mathbb{C}}
\providecommand{\eps}{\varepsilon}
\providecommand{\ov}{\overline}

\providecommand{\skp}[2]{\langle#1,#2\rangle}
\providecommand{\bigskp}[2]{\left\langle#1,#2\right\rangle}
\providecommand{\fR}{\mathfrak{R}}

\DeclareMathOperator{\supp}{supp}
\DeclareMathOperator{\curl}{curl}
\DeclareMathOperator{\sign}{sign}

\DeclareMathOperator{\diver}{div}

\DeclareMathOperator{\id}{id}
\DeclareMathOperator{\pv}{p.v.}

\DeclareMathOperator{\Id}{Id}
\DeclareMathOperator{\Real}{Re}
\DeclareMathOperator{\Imag}{Im}

\renewcommand{\qed}{\hfill $\Box$}

\setitemize{itemsep=+2pt}
\setenumerate{itemsep=+2pt}

\begin{document}

\allowdisplaybreaks

\title[Solutions for nonlinear Helmholtz and curl-curl equations]{Uncountably many solutions
for nonlinear Helmholtz and curl-curl equations with general nonlinearities}

\author{Rainer Mandel}
\address{R. Mandel \hfill\break
Karlsruhe Institute of Technology \hfill\break
Institute for Analysis \hfill\break
Englerstra{\ss}e 2 \hfill\break
D-76131 Karlsruhe, Germany}
\noindent\email{Rainer.Mandel@kit.edu}
\date{\today}


\begin{abstract}
  We obtain uncountably many solutions of nonlinear Helmholtz and curl-curl equations on the
  entire space using a fixed point approach. As an auxiliary tool a Limiting Absorption Principle for the
  curl-curl operator is proved.
\end{abstract}

\maketitle
\allowdisplaybreaks

\section{Introduction and main results}

The propagation of light in nonlinear media is  governed by Maxwell's equations
\begin{align*}
  &\nabla\times \mathcal E + \partial_t \mathcal B =0, &&\hspace{-3cm}\diver(D)=0, \\
  &\nabla\times \mathcal H - \partial_t \mathcal D =0, &&\hspace{-3cm}\diver(B)=0 
\end{align*}
for the electric respectively magnetic field $\mathcal E,\mathcal H:\R^3\times\R\to\R^3$, the displacement
field $\mathcal D:\R^3\times\R\to\R^3$ and the magnetic induction $\mathcal B:\R^3\times\R\to\R^3$. 
Here, the effect of charges and currents is neglected. The nonlinearity of the medium is typically expressed
through nonlinear material laws of the form $\mathcal D=\eps(x)\mathcal E+\mathcal P$ and the linear relation
$\mathcal H=\frac{1}{\mu}\mathcal B$ for the permittivity function $\eps:\R^3\to\R$ and the magnetic
permeability $\mu\in\R\sm\{0\}$. In~\cite{Med_GroundStates,BaDoPlRe_GroundStates} it was shown that special
solutions of the form $\mathcal E(x,t)=E(x)\cos(\omega t)$, $\mathcal P(x,t)=P(x,E)\cos(\omega t)$  can be approximately
described by solutions of the nonlinear curl-curl equations 
\begin{equation}\label{eq:NLM_allgemein}
  \nabla\times\nabla\times E + V(x)E = f(x,E) \qquad \text{in }\R^3
\end{equation}
where $V(x)=-\mu\omega^2\eps(x)\leq 0$ and $f(x,E)=\mu\omega^2 P(x,E)$. We stress that $V$ is nonpositive and
it becomes a negative constant in the simplest and most relevant case of the vacuum where $\eps(x)\equiv
\eps_0>0$. We refer to Section~1.3 in~\cite{BaDoPlRe_GroundStates} for further details concerning the
modelling aspect of \eqref{eq:NLM_allgemein}.
One of our main results (Theorem~\ref{thm:NLM}) will provide a new existence result for solutions of this problem. A
simplified version of~\eqref{eq:NLM_allgemein} is the nonlinear Helmholtz equation
\begin{equation} \label{eq:NLH_allgemein}
  -\Delta u + V(x) u = f(x,u) \qquad\text{in }\R^n,
\end{equation}
which  in the two-dimensional case $n=2$ can be derived from~\eqref{eq:NLM_allgemein} via the ansatz 
$E(x_1,x_2,x_3)=(0,0,u(x_1,x_2))$. In this paper we are interested  in
solutions of~\eqref{eq:NLM_allgemein},\eqref{eq:NLH_allgemein} when the potential $V$ is a negative constant
and the nonlinearity is rather general. Our principal motivation is to show that for large classes of
nonlinearities there are uncountably many solutions of these equations sharing the same decay rate
$|x|^{\frac{1-n}{2}}$ as $|x|\to\infty$ but with a different farfield pattern.

\medskip 

We first recall some facts about the nonlinear Helmholtz equation with constant
potential
\begin{equation} \label{eq:NLH}
  -\Delta u   - \lambda u = f(x,u) \qquad\text{in }\R^n.
\end{equation}
In 2004 Guti\'{e}rrez~\cite{Gut_nontrivial} set up a fixed point approach for this equation when
$f(x,u)=|u|^2u$, $n\in\{3,4\}$ and $u$ is complex-valued. Using an $L^p$-version of the Limiting Absorption
Principle for the Helmholtz operator (Theorem~6 in~\cite{Gut_nontrivial}) she found that small
nontrivial solutions of~\eqref{eq:NLH} can be obtained via the Contraction Mapping Theorem (Banach's Fixed
Point Theorem) on a small ball in $L^4(\R^n)$. Around ten years later 
Ev\'{e}quoz and Weth started to write a series of papers
\cite{Ev_dual,Eveq_plane,EvWe_branch,EvWe_dual,EvWe_real,EvYe_DualGS} containing new methods to prove existence results
for solutions of~\eqref{eq:NLH} that, in contrast to Guti\'{e}rrez' solutions, are large in suitable norms.
Some of these results were extended by the author in~\cite{MaMoPe_oscillating,Man_LAPperiodic}. In each of the aforementioned papers the nonlinearity
has to satisfy quite specific conditions that allow to deal with slow decay rates of solutions at infinity.

\medskip

In the case of power-type nonlinearities $f(x,u)=Q(x)|u|^{p-2}u$ one of the main results
in~\cite{EvWe_dual} is that there is an unbounded sequence of solutions in $L^p(\R^n)$ provided
$\frac{2(n+1)}{n-1}<p<\frac{2n}{n-2}$ and $Q\in L^\infty(\R^n)$ is positive and evanescent at infinity. If
$Q$ is $\Z^n$-periodic and positive (for negative $Q$ see Theorem~1.3,~1.4 in~\cite{MaMoPe_oscillating}) the
existence of one nontrivial solution is shown. These solutions are obtained using quite sophisticated dual
variational methods and the solution at the mountain pass level of the dual functional is called a dual ground
state of the equation. One of the drawbacks of this approach is that the assumption on $p$ does not allow for
cubic nonlinearities, which certainly are the most interesting ones for applications in physics. Moreover,
sign-changing or non-monotone nonlinearities can not be treated. In addition to that, solutions 
have to be looked for in $L^p(\R^n)$. This is a problem given that solutions decay
slowly at infinity so that a solution theory in $L^q(\R^n)$ with $q>p$ is more convenient a priori. 
For this reason we will not consider dual variational methods but rather revive Guti\'{e}rrez' fixed
point approach \cite{Gut_nontrivial}. 

\medskip

Our refinement of Guti\'{e}rrez' method allows to discuss nonlinear Helmholtz
equations with very general nonlinearities that improve existing results even in the case of power-type
nonlinearities as we will see below.  In our main result dealing with~\eqref{eq:NLH} we show that   in the
case $f(x,u)=Q(x)|u|^{p-2}u$ with $Q\in L^s(\R^n)\cap L^\infty(\R^n)$ we get solutions for all exponents
 \begin{equation} \label{eq:choice_of_ps}
    p > \max\Big\{ 2, \frac{2s(n^2+2n-1)-2n(n+1)}{(n^2-1)s} \Big\}.
\end{equation}
More generally, we can treat nonlinearities satisfying the following conditions:
 \begin{itemize}
  \item[(A)] $f:\R^n\times\R\to \R$ is a Carath\'{e}odory function satisfying for some 
  $Q\in L^s(\R^n)\cap L^\infty(\R^n)$
  \begin{align} \label{eq:conditions_f}
    \begin{aligned}
    |f(x,z)| &\leq Q(x)|z|^{p-1}  &&(x\in\R^n,|z|\leq 1)   \\
    |f(x,z_1)-f(x,z_2)| &\leq Q(x)(|z_1|+|z_2|)^{p-2}|z_1-z_2|  &&(x\in\R^n,|z_1|,|z_2|\leq 1).
    \end{aligned}   
  \end{align}
  where $s\in [1,\infty]$ and $p$ as in~\eqref{eq:choice_of_ps}. 
\end{itemize}
We stress that only conditions near zero are needed since we are going to construct small solutions in
$L^q(\R^n)$ which will turn out to be small also in $L^\infty(\R^n)$. Clearly, $|z|\leq 1$ can be replaced by
$|z|<z_0$ for any given $z_0>0$. We mention that in the case $s\leq \frac{n+1}{2}$ all exponents $p>2$ in the
superlinear regime are allowed.   

\medskip

The fundamental tools of Guti\'{e}rrez' fixed point approach are an $L^p$-version of the Limiting Absorption
Principle for the Helmholtz operator $-\Delta-\lambda$\,($\lambda>0$) and results about the so-called Herglotz
waves. As we will recall in Proposition~\ref{prop:Herglotz_waves}, these functions are analytic solutions of
the linear Helmholtz equation $-\Delta\phi-\lambda\phi=0$ in $\R^n$. They are given by the formula
\begin{equation*} 
    \widehat{h\,d\sigma_\lambda}(x) := \frac{1}{(2\pi)^{n/2}} \int_{S_\lambda^{n-1}} h(\xi)
    e^{-i\skp{x}{\xi}}\,d\sigma_\lambda(\xi) 
\end{equation*}
for complex-valued densities $h\in L^2(S^{n-1}_\lambda;\C)$. Here, $\sigma_\lambda$ denotes the canonical surface measure of the sphere
 $S_\lambda^{n-1}=\{\xi\in\R^n : |\xi|^2 = \lambda\}$. In order to ensure the real-valuedness and good
 poinwise decay properties of $\widehat{h\,d\sigma_\lambda}$ at infinity, we will consider a
 smaller class of densities $h$ belonging to the set 
 $$
   X_\lambda^\delta:= \Big\{h\in C^m(S_\lambda^{n-1};\C):h(\xi)=\ov{h(-\xi)},\|h\|_{C^m}\leq \delta \Big\}
 $$ 
 where $m:=\lfloor \frac{n-1}{2}\rfloor+1$. Here, our approach
 differs from \cite{Gut_nontrivial} where $L^2$-densities are used.   
 Theorem~\ref{thm:existence_via_BanachFPT} shows that for all    $h\in X_\lambda^\delta$ we
 find a strong solution of~\eqref{eq:NLH} that resembles
 $|x|^{\frac{1-n}{2}}u^\infty_h(x)$ at infinity where 
 $$
   u^\infty_h(x) := \lambda^{\frac{n-3}{4}} \sqrt{\frac{\pi}{2}}  
   \Real\left( e^{i(\frac{n-3}{4}\pi-\sqrt\lambda |x|)} \Big(  
    \widehat{f(\cdot,u_h)}(-\sqrt\lambda \hat x) + i\cdot 
    \frac{2\sqrt\lambda}{\pi} h(\sqrt\lambda \hat x)\Big)\right),\qquad 
    \hat x:=\frac{x}{|x|}.
 $$
 More precisely, we show the following.
 
\begin{thm} \label{thm:existence_via_BanachFPT}
  Assume (A) and $\lambda>0$. Then there are $\delta>0$ and mutually different solutions $(u_h)_{h\in
  X_\lambda^\delta}$ of~\eqref{eq:NLH} that form a $W^{2,r}(\R^n)$-continuum and satisfy
  $\|u_h\|_{W^{2,r}(\R^n)}\to 0$ as $\|h\|_{C^m}\to 0$ for any given $r\in (\frac{2n}{n-1},\infty)$ as well as
  $$
    \lim_{R\to\infty} \frac{1}{R} \int_{B_R} \left| u_h(x) -  
    |x|^{\frac{1-n}{2}} u_h^\infty(\hat x) \right|^2 \,dx = 0.
  $$
  If additionally $p>\frac{(3n-1)s-2n}{(n-1)s}$ holds, then $|u_h(x)|\leq C_h(1+|x|)^{\frac{1-n}{2}}$ for all
  $x\in\R^n$.
\end{thm}

  Let us discuss in which way this theorem improves earlier results. Most importantly,
  Theorem~\ref{thm:existence_via_BanachFPT} shows that nonlinear Helmholtz equations of the
  form~\eqref{eq:NLH} admit uncountably many solutions for a large class of nonlinearities which need not be
  odd, let alone of power-type. Its proof is short and elementary in the
  sense that it only uses the Contraction Mapping Theorem, elliptic regularity theory and
  mostly well-known results about the linear Helmholtz equation. Up to now such general nonlinearities have
  only been treated in the paper~\cite{EvWe_real} by Ev\'{e}quoz and Weth, but their additional requirement
  $(f_0)$ on p.361 requires the nonlinearity to be supported in a bounded subset of $\R^n$, which is quite
  restrictive. In our approach such an assumption is not necessary.  
  Given that applications often deal with power-type nonlinearities $f(x,z)=Q(x)|z|^{p-2}z$ let us
  comment on our improvements for this particular case in more detail. In the case $Q\in L^\infty(\R^n)$  we
  obtain solutions for all exponents $p>\frac{2(n^2+2n-1)}{n^2-1}$. This bound is smaller than $\frac{2(n+1)}{n-1}$
  so that our range of exponents is larger than in all other nonradial approaches except for~\cite{EvWe_real}
  where $Q$ has compact support and exponents $2<p<\frac{2n}{n-2}$ are allowed.  
  Additionally, we need not require $Q$ to be periodic nor evanescent (as
  in~\cite{Eveq_plane,EvWe_dual,MaMoPe_oscillating}) nor compactly supported (as
  in~\cite{EvWe_real,EvWe_branch}) and the growth rate of the nonlinearity may be supercritical (i.e. $p\geq
  \frac{2n}{n-2}$) which is an entirely new feature. The latter fact is worth mentioning given that
  Ev\'{e}quoz and Yesil~\cite{EvYe_DualGS} proved the nonexistence of dual ground states $u\in L^p(\R^n)$ for
  $n=3$ in the critical case $p=\frac{2n}{n-2}=6$ provided $f(x,u)=Q(x)u^5$ and $Q\in L^\infty(\R^3)$ is
  nonnegative and nontrivial. Since Theorem~\ref{thm:existence_via_BanachFPT} yields solutions belonging to
  $L^p(\R^3)$ we conclude that dual ground states need not exist while other nontrivial solutions do.  
  Finally we mention that in the physically most relevant case of a cubic nonlinearity $p=4$ we obtain
  uncountably many solutions whenever $n\geq 3,s\in [1,\infty]$ or $n=2,s\in [1,6)$.

\begin{bem} \label{Remark1} 
\begin{itemize}   
  \item[(a)] The decay rate $|x|^{\frac{1-n}{2}}$ is best possible. This is a consequence of
  Theorem~3 in~\cite{KoTa_Carleman} where nontrivial solutions of the elliptic PDE $-\Delta u -\lambda u = W(x)u$ in
  $\R^n$ with $W\in L^{\frac{n+1}{2}}(\R^n)$ and $\lambda>0$ are shown to satisfy
  $u(x)|x|^{-\frac{1}{2}-\eps}\notin L^2(\R^n)$ for all $\eps>0$. In particular,  better decay rates than
  $|x|^{\frac{1-n}{2}}$ as $|x|\to\infty$ are excluded.
  Notice that in the setting of Theorem~\ref{thm:existence_via_BanachFPT} the function $W(x):=f(x,u(x))/u(x)$
  satisfies  $W\in L^{\frac{n+1}{2}}(\R^n)$ because of $Q\in L^s(\R^n)$ and
  $$
    \qquad |W(x)|\leq Q(x)|u(x)|^{p-2},\quad
    u\in \bigcap_{r>\frac{2n}{n-1}} L^r(\R^n) ,\quad 
    p>\frac{2s(n^2+2n-1)-2n(n+1)}{(n^2-1)s},
  $$
  see~\eqref{eq:choice_of_ps}. It is remarkable that precisely this lower bound for $p$ appears in this
  context. Up to now existence and optimal decay results for nonlinear Helmholtz equations for lower exponents
  $p$ are only known in the radial setting \cite{MaMoPe_oscillating,EvWe_real}.
  Notice that for smaller $p$ the (nonradial) counterexample of Ionescu and Jerison from Theorem~2.5
  in~\cite{IonJer_absence} has to be taken into account: For any given $N\in\N$ there is $W\in L^q(\R^n)$
  with $q>\frac{n+1}{2}$ and a solution of $-\Delta u -\lambda u = W(x)u$ in $\R^n$ with $|u(x)|\leq
  (1+|x|)^{-N}$ for all $x\in\R^n$.
 \item[(b)] Theorem~\ref{thm:existence_via_BanachFPT} yields a symmetry-breaking result: 
 For any subgroup $\Gamma\subset O(n)$ such that
 $\Gamma\neq \{\id\}$ and any $\Gamma$-invariant nonlinearity $f$ satisfying (A) one has uncountably many
 solutions that are not $\Gamma$-invariant. In particular, for $\Gamma=O(n)$, radial nonlinearities allow for
 nonradial solutions. We will prove this in Remark~\ref{Remark2}(a).  
  \item[(c)] In order to construct radial solutions, we can get the same conclusions as in 
  Theorem~\ref{thm:existence_via_BanachFPT} under weaker assumptions on $p$. 
  This is due to an improved version of the Limiting Absorption Principle for the Helmholtz operator. In
  Remark~\ref{Remark2}(b) we comment on the necessary modifications of the proof and show that the admissible
  range of exponents for the existence of radial solutions is no longer given by \eqref{eq:choice_of_ps}, but
  \begin{equation}\label{eq:choicep_radial}
     p > \max\Big\{ 2,  \frac{s(2n^2+n-1)-2n^2}{sn(n-1)} \Big\}. 
  \end{equation}
  Notice that the resulting radial version of Theorem~\ref{thm:existence_via_BanachFPT} is not covered by
  earlier contributions from~\cite{MaMoPe_oscillating} (Theorem~1.2, Theorem~2.10) or~\cite{EvWe_real} (Theorem~4). For
  instance, it provides solutions of the radial nonlinear Helmholtz equation 
  $-\Delta u - u = Q(x)|u|^{p-2}u$ for any $Q\in L^s_{rad}(\R^n)\cap L^\infty_{rad}(\R^n)$ and $p$ as
  in~\eqref{eq:choicep_radial} whereas in the above-mentioned papers $Q$ has to be bounded,
  differentiable and radially decreasing. On the other hand, our restrictions on the exponent $p$ do not
  appear in \cite{MaMoPe_oscillating,EvWe_real} (where all $p>2$ are allowed) so
  that~\eqref{eq:choicep_radial} might be improved further.
 \end{itemize}
\end{bem}

  Next we discuss variants of these results for related semilinear elliptic PDEs from mathematical
  physics. First let us mention that a nonlinearity $f(\cdot,u)$
  satisfying (A) may be without any major difficulty be replaced by a nonlocal right hand side such
  as $K\ast f(\cdot,u)$ where $K\in L^1(\R^n)$. Clearly, imposing more assumptions
  $K$ may even lead to larger ranges of exponents than~\eqref{eq:choice_of_ps}. In this way it is possible to
  obtain small solutions of nonlocal Helmholtz equations. Similarly, one may ask how our results are affected
  by changes in the linear operator. For instance, if the Helmholtz operator is perturbed to a periodic
  Schr\"odinger operator $-\Delta+V(x)-\lambda$ then it should be possible to adapt the proof in such a way
  that it provides small solutions of $-\Delta u +V(x)u -\lambda u = f(x,u)$ in $\R^n$ provided $\lambda$
  belongs to the essential spectrum of $-\Delta+V(x)$ and the band structure of this periodic Schr\"odinger
  operator is sufficiently nice. To be more precise, one would require (A1),(A2),(A3) from~\cite{Man_LAPperiodic} to
  hold so that Herglotz-type waves, defined as suitable oscillatory integrals over the so-called Fermi
  surfaces associated with $-\Delta+ V(x)$, exist and have the properties stated in
  Proposition~\ref{prop:Herglotz_waves} below. Since the technicalities (including
  a Limiting Absorption Principle for such operators) are quite involved and mostly carried out in
  \cite{Man_LAPperiodic}, we prefer not to discuss this issue further.

  \medskip
  
  We now turn our attention to a fourth order version of~\eqref{eq:NLH}  given by
  \begin{equation}\label{eq:4thorderNLH}
    \Delta^2 u - \beta\Delta u + \alpha u = f(x,u) \quad\text{in }\R^n,
  \end{equation}
  which we will briefly discuss for $\alpha,\beta$ satisfying  
  \begin{equation}\label{eq:case_distinction}
    \text{(i)}\;\; \alpha<0,\beta\in\R\qquad\quad\text{or}\qquad\quad 
    \text{(ii)}\;\; \alpha>0,\beta<-2\sqrt\alpha.
  \end{equation}
  Under these assumptions dual variational methods were employed in~\cite{BoCaMa_4thorder} to prove the
  existence of one nontrivial solution when $f(x,z)=Q(x)|z|^{p-2}z$ where
  $\frac{2(n+1)}{n-1}<p<\frac{2n}{n-2}$ and $Q$ is positive and $\Z^n$-periodic.
  Notice that in the case $\beta^2-4\alpha<0$ classical variational methods such as
  constrained minimization apply and a number of papers revealed the existence of positive and sign-changing
  solutions $u\in H^4(\R^n)$ of~\eqref{eq:4thorderNLH} again for power-type nonlinearities. We refer
  to \cite{BoCadoSNa_orbitally,BoNa_waveguide,BoCaGouJe_normalized} for
  results in this direction. Our intention is to show that in the case (i) or (ii) uncountably many
  solutions of~\eqref{eq:4thorderNLH} exist for all nonlinearities $f$ satisfying (A).
  The main observation is that in case (i) or (ii) there are analoga of the Herglotz waves given by densities
  $h\in Y^\delta$ where
  \begin{align} \label{eq:Ydelta}
    \begin{aligned}
    \text{In case (i)\,:}\quad Y^\delta &:= X_\lambda^\delta,  \hspace{1.9cm}\text{where}\qquad 
    \lambda = \frac{-\beta+\sqrt{\beta^2-4\alpha}}{2}>0,\\
    \text{In case (ii):}\quad Y^\delta &:= X_{\lambda_1}^\delta \times X_{\lambda_2}^\delta \qquad 
    \text{where}\quad
    \lambda_{1,2} = \frac{-\beta\pm \sqrt{\beta^2-4\alpha}}{2}>0. 
  \end{aligned}
  \end{align}
  The fixed point approach used in the proof of Theorem~\ref{thm:existence_via_BanachFPT} may be rather
  easily adapted to~\eqref{eq:4thorderNLH} and we can prove the following result. 
   
  \begin{thm} \label{thm:4thorder}
  Assume (A) and (i) or (ii). Then there is $\delta>0$ and mutually different solutions $(u_h)_{h\in
  Y^\delta}$ of~\eqref{eq:4thorderNLH} that form a $W^{4,r}(\R^n)$-continuum satisfying
  $\|u_h\|_{W^{4,r}(\R^n)}\to 0$ as $\|h\|_{C^m}\to 0$ for any given $r\in (\frac{2n}{n-1},\infty)$.
  If additionally $p>\frac{(3n-1)s-2n}{(n-1)s}$ holds, then $|u_h(x)|\leq C_h(1+|x|)^{\frac{1-n}{2}}$ for all
  $x\in\R^n$.
  \end{thm} 

  As in Theorem~\ref{thm:existence_via_BanachFPT} one can say more about the asymptotics  of
  the constructed solutions; we refer to Section~5.2 in~\cite{BoCaMa_4thorder} for a related discussion.
  Further generalizations to more general higher order semilinear elliptic problems of the form $Lu = f(x,u)$
  in $\R^n$ are possible provided   the linear differential operator with constant coefficients $L$ has
  a Fourier symbol $P(\xi)$ with the property that $\{\xi\in\R^n: P(\xi)=0\}$ is a compact manifold
  with nonvanishing Gaussian curvature. Notice that this assumption makes the method of stationary phase work
  and provides pointwise decay of oscillatory integrals as demonstrated in the proof of
  Proposition~\ref{prop:Herglotz_waves}. Moreover, one needs a Limiting
  Absorption Principle in order to make sense of the Fourier multiplier $1/P(\xi)$ as a mapping between
  Lebesgue spaces. At least in the case $P(\xi)=P_0(\xi)(|\xi|^2-\lambda_1)\cdot\ldots
  \cdot (|\xi|^2-\lambda_k)$ with $0<\lambda_1<\ldots<\lambda_k$ and $P_0$ positive this can be established as
  in Theorem~3.3 in~\cite{BoCaMa_4thorder}. With these tools our fixed point approach can be adapted to find nontrivial
  solutions of $Lu=f(x,u)$ in $\R^n$.

  \medskip
  
  Finally, we discuss nonlinear curl-curl equations of the form
  \begin{equation} \label{eq:NLM}
    \nabla \times \nabla \times E - \lambda E =   f(x,E) \quad\text{in }\R^3
  \end{equation}
  that describe the electric field $E:\R^3\to\R^3$ of an electromagnetic wave in a nonlinear medium.  
  This equation has been studied in the past years on bounded
  domains in $\R^3$ \cite{BaMe_NonlinearTimeharmonic,BaMe_NonlinearTimeharmonicAnisotropic} but also on the entire space
  $\R^3$, which is the situation we focus on. Up to our knowledge there is only one result    
  for solutions of nonlinear curl-curl equations on $\R^3$ without symmetry assumption.
  In \cite{Med_GroundStates} Mederski proves the existence of a weak solution of \eqref{eq:NLM} by variational methods
  when $\lambda$ is replaced by a small nonnegative potential $V(x)$ that decays suitably
  fast to zero at infinity, see assumption (V) in \cite{Med_GroundStates}. In particular, constant functions
  $V(x)=\lambda$ can not be treated by this method so that our setting must be considered as entirely
  different from the one in~\cite{Med_GroundStates}.
  
  \medskip
  
  In the cylindrically symmetric setting the existence of solutions can be proved using various approaches.
  Here, the electrical field is assumed to be of the form 
  \begin{equation}\label{eq:cylindricallysym}
    E(x_1,x_2,x_3) = \frac{E_0(\sqrt{x_1^2+x_2^2},x_3)}{\sqrt{x_1^2+x_2^2}} \vecIII{-x_2}{x_1}{0}
    \qquad\text{where }E_0:[0,\infty)\times\R\to\R. 
  \end{equation}
  Such functions are divergence-free so that $\nabla\times\nabla\times E = -\Delta E$ implies that one
  actually has to deal with the elliptic $3\times 3$-system
  \begin{equation} \label{eq:NLM_cylindrical}
    - \Delta E - \lambda E = f(x,E) \quad\text{in }\R^3,
  \end{equation}
  which may equally be expressed in terms of $E_0$ provided the nonlinearity $f(x,E)$ is compatible with this
  symmetry assumption, see page 3 in~\cite{BaDoPlRe_GroundStates}. In this special case
  further results \cite{HirRei_Existence,BaDoPlRe_GroundStates,Zeng_CylindricallySym} are known but none of
  those applies in the case $\lambda>0$ and $f(x,E)=\pm |E|^{p-2}E$ that we are mainly interested in.
  
  \medskip
   
  Given our earlier results for the nonlinear Helmholtz equation~\eqref{eq:NLH} it is not surprising that we
  obtain an existence result for \eqref{eq:NLM_cylindrical} that is entirely analogous to the
  one from Theorem~\ref{thm:existence_via_BanachFPT}. Since this result fills a gap in the literature, we
  state it in part (i) of our theorem  even though its proof is a straightforward
  adaptation of the fixed point approach used in the proof of Theorem~\ref{thm:existence_via_BanachFPT}. The
  corresponding assumption on the nonlinearity is the following.
  \begin{itemize}
    \item[(A')]   $f:\R^3\times\R^3\to \R^3,(x,E)\mapsto f_0(\sqrt{x_1^2+x_2^2},x_3,|E|^2)E$ is a
    Carath\'{e}odory function satisfying~\eqref{eq:conditions_f}  for some $Q\in  L^s(\R^3)\cap L^\infty(\R^3)$.
  \end{itemize}
  Assumption (A') ensures that $f$ is compatible with cylindrical symmetry. Indeed, for $E$ as
  in~\eqref{eq:cylindricallysym} one can check that $f(\cdot,E)$ is of the form~\eqref{eq:cylindricallysym},
  too.
  
  \medskip
  
  In the general non-symmetric case the construction of solutions is more difficult since the curl-curl
  operator satisfies a much weaker Limiting Absorption Principle as in the cylindrically symmetric
  setting, cf. Theorem~\ref{thm:LAP_NLM}. Moreover, well-known regularity results for elliptic problems are
  not available so that we have to consider a substantially smaller class of nonlinearities satisfying the
  following:
\begin{itemize}
  \item[(B)]   $f:\R^3\times\R^3\to \R^3$ is a Carath\'{e}odory function satisfying for some 
  $Q\in L^s(\R^3)\cap L^\infty(\R^3)$  the estimate 
  \begin{align}\label{eq:conditions_f_NLM}
    \begin{aligned}
    |f(x,E)| &\leq  Q(x)|E|^{p-1}(1+|E|)^{\tilde p-p} &&(x,E\in\R^3), \\
    |f(x,E_1)-f(x,E_2)| &\leq Q(x)(|E_1|+|E_2|)^{p-2}(1+|E_1|+|E_2|)^{\tilde p-p}|E_1-E_2| 
    &&(x,E_1,E_2\in\R^3)
    \end{aligned}   
  \end{align}
  where $1\leq s\leq 2$ and $\tilde p\leq 2\leq p<\infty$.  
\end{itemize}
  Additionally, we will have to require that $\|Q\|_s+\|Q\|_\infty$ is small enough
  in order to obtain solutions of~\eqref{eq:NLM}. In the cylindrically symmetric respectively
  non-symmetric setting the counterparts of the Herglotz waves (introduced in Section~\ref{sec:ResLAP}) are parametrized by functions
  $h\in Z^\delta_{cyl}$ respectively $h\in Z$ where
  \begin{align*}
    Z 
    &:= \Big\{ h\in C^2(S_\lambda^2;\C^3): h(\xi)=\ov{h(-\xi)},\,\skp{h(\xi)}{\xi}=0\; \forall\xi\in
    S_\lambda^2 \Big\}, \\
    Z^{\delta}_{cyl} 
    &:= \Big\{ h\in Z: \|h\|_{C^2} <\delta \text{ and } \Real(h),\Imag(h) \text{ satisfy
    \eqref{eq:cylindricallysym}}\Big\}.
  \end{align*}
  Notice that both sets are nonempty. With these definitions we can formulate our main result for
  the nonlinear curl-curl equation~\eqref{eq:NLM}.

\begin{thm} \label{thm:NLM} ~
  \begin{itemize}
    \item[(i)] Assume (A') and  $\lambda>0$.  Then there is $\delta>0$ and  a family $(E_h)_{h\in
    Z_{cyl}^\delta}$ of mutually different cylindrically symmetric solutions of~\eqref{eq:NLM} that form a
    $W^{2,r}(\R^3;\R^3)$-continuum and satisfy $\|E_h\|_{W^{2,r}(\R^3;\R^3)}\to 0$ as $\|h\|_{C^2}\to 0$ for
    any given $r\in (3,\infty)$. If additionally $p>\frac{4s-3}{s}$ holds, then
    $|E_h(x)|\leq C_h(1+|x|)^{-1}$.
    \item[(ii)] Assume (B) and $\lambda>0,3<q<\frac{3s}{(2s-3)_+}$. If  $\|Q\|_s+\|Q\|_\infty$ is sufficiently
    small then there is a family $(E_h)_{h\in Z}$ of mutually different weak solutions of~\eqref{eq:NLM} lying in
    $H_{loc}(\curl;\R^3)\cap L^q(\R^3;\R^3)$. Moreover:
    \begin{itemize}
      \item[(a)] If additionally $(p,s)\neq (2,2)$ holds, then $E_h\in L^r(\R^3;\R^3)$ for all $r\in (3,q)$.
      \item[(b)] If additionally $\tilde p<2<p$ holds, then $E_h\in L^r(\R^3;\R^3)$ for all $r\in
      (q,\frac{3s(p-1)}{(2s-3)_+})$.
    \end{itemize}   
  \end{itemize}
\end{thm}

As an application we obtain uncountably many distinct weak solutions of the curl-curl equation
\eqref{eq:NLM} with saturated nonlinearities of the form 
$$
  f(x,E)=  \frac{\delta|E|^2 \Gamma(x)E}{1+P(x)|E|^2} 
$$
where $\inf P>0$, $\Gamma\in L^s(\R^3;\R^{3\times 3})\cap L^\infty(\R^3;\R^{3\times3})$  
and $\delta>0$ is sufficiently small. 

\medskip

The paper is organized as follows. In Section~\ref{sec:ResLAP} we review the Limiting
Absorption Principles  and properties of Herglotz-type waves that we will need for the proofs of our
results. In  Section~\ref{sec:ProofThm1}, Section~\ref{sec:ProofThm2},
Section~\ref{sec:ProofThm3} we then prove Theorem~\ref{thm:existence_via_BanachFPT},
Theorem~\ref{thm:4thorder} and Theorem~\ref{thm:NLM}. Since the proofs of
Theorem~\ref{thm:existence_via_BanachFPT}, Theorem~\ref{thm:4thorder} and Theorem~\ref{thm:NLM}~(i) are 
almost identical, we carry out the first in detail and only comment on the modifications when it comes to
the latter. Finally, in Appendix A we review the method of stationary phase and prove
Proposition~\ref{prop:LAP_asymptotics}. In Appendix~B we  prove our Limiting Absorption Principle for the
curl-curl operator (Theorem~\ref{thm:LAP_NLM}). In Appendix C we review some resolvent estimates for the Helmholtz operator due
to Ruiz and Vega.

\medskip

 In the following $C$ will denote a generic constant that can change from line to line and $\frac{1}{r_+}$
 stands for $\frac{1}{r}$ if $r>0$ and for $\infty$ if $r\leq 0$.
 The symbol $\mathcal F f=\hat f$ represents the Fourier transform of (the tempered distribution)
 $f\in L^q(\R^n)$ and $\mathcal F_1,\mathcal F_{n-1}$ are the Fourier transforms in $\R^1,\R^{n-1}$, respectively. For $R>0$ the symbol
 $B_R$ denotes the open ball of radius $R$ around the origin in $\R^n$ and
 $\skp{\cdot}{\cdot}$ is the inner product in $\R^n$ extended by bilinearity to $\C^n$.
 $L^q(\R^n),L^{q,w}(\R^n)$ denote the classical respectively weak Lebesgue spaces on $\R^n$
 equipped with the standard norms $\|\cdot\|_q,\|\cdot\|_{q,w}$.


\section{Herglotz waves and Limiting absorption principles} \label{sec:ResLAP}
  
  In this section we review some partly well-known results on Herglotz waves and Limiting Absoprtion
  Principles for the linear differential operators we are interested in. A classical Herglotz wave associated with the
  Helmholtz operator $-\Delta-\lambda$ is defined via the formula
  $$
    \mathcal F(h\,d\sigma_\lambda)(x)
    := \widehat{h\,d\sigma_\lambda}(x) 
    := \frac{1}{(2\pi)^{n/2}} \int_{S_\lambda^{n-1}} h(\xi) e^{-i\skp{x}{\xi}}\,d\sigma_\lambda(\xi) 
  $$
  where $h\in L^2(S_\lambda^{n-1};\C)$ and $\sigma_\lambda$ denotes the canonical surface measure of 
  $S_\lambda^{n-1}=\{\xi\in\R^n : |\xi|^2 = \lambda\}$.
   Herglotz waves are analytic functions that solve the linear Helmholtz equation
 $-\Delta\phi-\lambda\phi=0$.  Their pointwise decay properties are well-understood for smooth densities $h$
 and result from an application of the method of stationary phase. Unfortunately, we could not find a
 quantitative version  of this result telling how smooth the density $h$ needs to be in order to ensure that
 $\widehat{h\,d\sigma_\lambda}(x)$ decays like $|x|^{\frac{1-n}{2}}$  as
 $|x|\to\infty$ in the pointwise sense. In our first auxiliary result we provide such an estimate and its
 proof will be given in  Appendix A. For notational convenience we introduce the quantity
 \begin{equation}\label{eq:def_mh}
   m_h(x):=  e^{i(\frac{n-1}{4}\pi-\sqrt\lambda |x|)}h(\sqrt\lambda \hat x)
      + e^{-i(\frac{n-1}{4}\pi-\sqrt\lambda |x|)} 
       h(-\sqrt\lambda \hat x) 
 \end{equation}
 so that our claim is the following.
  
 \begin{prop} \label{prop:Herglotz_waves}
   Let $n\in\N,n\geq 2$ and $m:= \lfloor \frac{n-1}{2}\rfloor+1$. Then for all $h\in
   C^m(S_\lambda^{n-1};\C)$ the Herglotz wave $\widehat{h\,d\sigma_\lambda}$ is an analytic solution of
   $-\Delta\phi-\lambda\phi=0$ in $\R^n$ and satisfies the estimate $|(\widehat{h\,d\sigma_\lambda)}(x)|\leq
   C\|h\|_{C^m}(1+|x|)^{\frac{1-n}{2}}$ as well as 
   \begin{equation*}
     \lim_{R\to\infty} \frac{1}{R} \int_{B_R} \left| \widehat{h\,d\sigma_\lambda}(x) - 
      \frac{1}{\sqrt{2\pi}} \left(\frac{\sqrt\lambda}{|x|}\right)^{\frac{n-1}{2}}      
      m_h(x)  \right|^2\,dx
       = 0.
   \end{equation*}
   In particular, we have $\|\widehat{h\,d\sigma_\lambda}\|_r\leq C_r \|h\|_{C^m}$ for all $r>\frac{2n}{n-1}$. 
 \end{prop}
%
%

While Herglotz waves solve the homogeneous Helmholtz equation, we also need to discuss the
inhomogeneous equation. Since $\lambda$ lies in the essential spectrum of $-\Delta$ it is a nontrivial task to
solve $-\Delta u - \lambda u =f$ in $\R^n$. The method to find such solutions is to study the limit
of solutions $u_\eps:=\mathcal R(\lambda+i\eps)f\in H^2(\R^n;\C)$ of $-\Delta u_\eps - (\lambda+i\eps)u_\eps
= f$ in a suitable topology. The complex-valued limit of these functions as $\eps\to 0^+$ is denoted by
$\mathcal R(\lambda+i0)f$ and we define $\fR_\lambda f:=\Real(\mathcal R(\lambda+i0))f$ as its real part since
we are interested in real-valued solutions. These operators have the following properties:
 
\begin{thm}[Theorem~6~\cite{Gut_nontrivial}, Theorem~2.1~\cite{Eveq_plane}] \label{thm:LAP}
  Let $n\in\N,n\geq 2$. The operator $\fR_\lambda:L^t(\R^n)\to L^q(\R^n)$ is a bounded linear operator provided 
 \begin{align} \label{eq:admissible}
   \begin{aligned}
     \frac{1}{t}>\frac{n+1}{2n},\qquad
     \frac{1}{q}<\frac{n-1}{2n},\qquad  
     \frac{2}{n+1}\leq \frac{1}{t}-\frac{1}{q} \leq  \frac{2}{n} \qquad (n\geq 3), \\
     \frac{1}{t}>\frac{n+1}{2n},\qquad
     \frac{1}{q}<\frac{n-1}{2n},\qquad  
     \frac{2}{n+1}\leq \frac{1}{t}-\frac{1}{q} <  \frac{2}{n} \qquad (n=2).
 \end{aligned}
 \end{align}
 Moreover, for $f\in L^t(\R^n)$ the function $\fR_\lambda f\in W^{2,t}_{loc}(\R^n)$ is a real-valued strong
 solution of $-\Delta u - \lambda u=f$ in $\R^n$.
\end{thm}

The last statement is actually not included in the references given above, but it is a consequence of elliptic
regularity theory for distributional solutions. We refer to Proposition~A.1 in~\cite{EvWe_dual} for 
a similar result. Next we discuss the asymptotic behaviour of the solutions $\fR_\lambda f$ that we will
deduce from the following result.

  \begin{prop}\label{prop:LAP_asymptotics}
    Let $n\in\N,n\geq 2$ and assume $f\in L^{p'}(\R^n)$ for $\frac{2(n+1)}{n-1}\leq p\leq
    \frac{2n}{(n-4)_+},(n,p)\neq (4,\infty)$.
    Then:
     $$ 
     \lim_{R\to\infty}  \frac{1}{R} \int_{B_R} \left| \mathcal R(\lambda+i0)f(x) -  
     \sqrt{\frac{\pi}{2\lambda}} \left(\frac{\sqrt\lambda}{|x|}\right)^{\frac{n-1}{2}}
     e^{i(\frac{n-3}{4}\pi-\sqrt\lambda|x|)} \widehat{f}(-\sqrt\lambda \hat x)  \right|^2\,dx 
     = 0. 
   $$
  \end{prop}
  \begin{proof}
    The claim for $n\geq 3$ and $\lambda=1$ is provided in Proposition~2.7~\cite{EvWe_dual} 
    so that the general case follows from rescaling via $\mathcal{R}_\lambda f(x) = \frac{1}{\lambda}
    \mathcal{R}_1 (f(\lambda^{-1/2}\cdot))(\sqrt\lambda x)$. The proof in the case $n=2$ is essentially the same.
    Indeed, repeating the proof of Proposition~2.7~\cite{EvWe_dual} one finds that the claimed result holds
    true provided Proposition~2.6 in~\cite{EvWe_dual} (the Stein-Tomas Theorem) 
    and the estimate 
    \begin{equation}\label{eq:estimate_asymptotics}
      \sup_{R\geq 1} \left(\frac{1}{R}\int_{B_R} |\mathcal R(\lambda+i0)f(x)|^2\,dx  \right)^{1/2} \leq C
      \|f\|_{p'} 
    \end{equation}
    are valid in the case $n=2$ under our assumptions on $p$. For the Stein-Tomas Theorem this is clear. 
    The inequality~\eqref{eq:estimate_asymptotics} is due to Ruiz and Vega~\cite{RuVe_OnLocal}, but we could
    not find an accurate reference for it that covers our range of
    exponents and all space dimensions $n\geq 2$. We provide the estimate~\eqref{eq:estimate_asymptotics}
    and further details in Appendix~C so that the proof is finished.
  \end{proof}

  Next we recall a Limiting Absorption Principle that we will need in the discussion of the fourth order
  problem~\eqref{eq:4thorderNLH}. In Theorem~3.3 in~\cite{BoCaMa_4thorder} the following
  extension of Theorem~\ref{thm:LAP} to linear differential operators of the
  form $\Delta^2-\beta\Delta+\alpha$ was proved.

\begin{thm}[Theorem~3.3~\cite{BoCaMa_4thorder}] \label{thm:4thorder_resolvent}
  Let $n\in\N,n\geq 2$ and assume (i) or (ii). Then there is a bounded linear operator
  $\bold{R}:L^t(\R^n)\to L^q(\R^n)$ for   
   \begin{align*}  
   \begin{aligned}
     \frac{1}{t}>\frac{n+1}{2n},\qquad
     \frac{1}{q}<\frac{n-1}{2n},\qquad  
     \frac{2}{n+1}\leq \frac{1}{t}-\frac{1}{q} \leq  \frac{4}{n} \qquad &\text{if }n\geq 5, \\
     \frac{1}{t}>\frac{n+1}{2n},\qquad
     \frac{1}{q}<\frac{n-1}{2n},\qquad  
     \frac{2}{n+1}\leq \frac{1}{t}-\frac{1}{q} <  1 \qquad &\text{if }n=4, \\
     \frac{1}{t}>\frac{n+1}{2n},\qquad
     \frac{1}{q}<\frac{n-1}{2n},\qquad  
     \frac{2}{n+1}\leq \frac{1}{t}-\frac{1}{q} \leq 1 \qquad &\text{if }n=2,3
 \end{aligned}
 \end{align*}
 such that for $f\in L^t(\R^n)$ the function $\bold{R} f$ belongs to $W^{4,t}_{loc}(\R^n)$ and is a
 real-valued strong solution of $\Delta^2 u-\beta\Delta u +\alpha u=0$ in $\R^n$.
\end{thm}

  Finally we provide the tools for proving Theorem~\ref{thm:NLM}. As for the previous results we need a
   family of elements lying in the kernel of the linear operator which now is  
   $E\mapsto \nabla\times\nabla\times E-\lambda E$. These are given by  vectorial
   variants of the Herglotz waves 
   $$
    \widehat{h\,d\sigma_\lambda}(x) 
    :=  \frac{1}{(2\pi)^{3/2}} \int_{S_\lambda^2} h(\xi) e^{-i\skp{x}{\xi}}\,d\sigma_\lambda(\xi) 
   $$
   (the integral to be understood componentwise) where $h:S_\lambda^2\to\C^3$ is a tangential vectorfield
   field, i.e. $\skp{h(\xi)}{\xi}=0$ for all $\xi\in S_\lambda^2$. These
   functions are real-valued whenever $h(\xi)=\ov{h(-\xi)}$. Applying the results from
   Proposition~\ref{prop:Herglotz_waves} in each component, we deduce the following properties.
 
 \begin{prop}\label{prop:Herglotz_waves_NLM}
   For all $h\in Z$ the function $ \widehat{h\,d\sigma_\lambda}$  is an analytic solution of
   $\nabla\times\nabla\times \phi - \lambda\phi=0$ in $\R^3$ and satisfies
   the pointwise estimate $|\widehat{h\,d\sigma_\lambda}(x)|\leq \|h\|_{C^2}(1+|x|)^{-1}$ for all $x\in\R^3$.
   In particular, $\|\widehat{h\,d\sigma_\lambda}\|_r \leq C_r \|h\|_{C^2}$ for all $r>3$. 
   If $h\in Z^\delta_{cyl}$, then $\widehat{h\,d\sigma_\lambda}$ is cylindrically symmetric.
 \end{prop}
 
  Having described the analoga of the Herglotz waves we finally discuss a  Limiting Absorption
  Principle for the curl-curl operator. So let $\mathcal R(\lambda+i\eps)$
  denote the resolvent of 
  $E\mapsto \nabla\times\nabla\times E - (\lambda+i\eps)E$ which we will prove to exist in
  Proposition~\ref{prop:selfadjointness}. As for the Helmholtz operator one is interested in the
  (complex-valued) limit $\mathcal R(\lambda+i0)G$ for $G\in L^2(\R^3;\R^3)$.
  It turns out that $G$ decomposes into two parts behaving quite differently. So we split
  $G$ into a curl-free (gradient-like) part $G_1:\R^3\to\R^3$ and a divergence-free remainder
  $G_2:\R^3\to\R^3$ of $G$ which are defined via
  $$ 
   G_1 := \mathcal F^{-1}\left(\skp{\hat G(\xi)}{\frac{\xi}{|\xi|}}\frac{\xi}{|\xi|}\right),\qquad 
   G_2 := \mathcal F^{-1}\left(\hat G(\xi)-\skp{\hat G(\xi)}{\frac{\xi}{|\xi|}}\frac{\xi}{|\xi|}\right).
  $$
  This splitting corresponds to a Helmholtz decomposition of a vector field in $\R^3$. 

\begin{thm} \label{thm:LAP_NLM}
  Let $\lambda>0$ and assume that $t,q\in (1,\infty)$ satisfy~\eqref{eq:admissible}. Then there is a bounded
  linear operator $R_\lambda:L^t(\R^3;\R^3)\cap L^q(\R^3;\R^3)\to L^q(\R^3;\R^3)$ such that $R_\lambda G\in
  H_{loc}(\curl;\R^3)$ is a weak solution of $\nabla\times\nabla\times E-\lambda E = G$ 
  provided $G\in L^t(\R^3;\R^3)\cap L^q(\R^3;\R^3)$. Moreover, we have
  \begin{equation*}
    \|R_\lambda G\|_q \leq C(\|G_1\|_q+\|G_2\|_t) \leq C(\|G\|_q+\|G\|_t) 
  \end{equation*} 
  and $R_\lambda G = -\frac{1}{\lambda}G_1 + \fR_\lambda G_2$ for $\fR_\lambda$ from Theorem~\ref{thm:LAP}
  (applied componentwise).
  If $G\in L^t(\R^3;\R^3)$ is cylindrically symmetric then so is $R_\lambda G$ and $R_\lambda G\in
  W_{loc}^{2,q}(\R^3;\R^3)$ is a strong solution satisfying $\|R_\lambda G\|_q \leq C\|G\|_t$.
\end{thm}

 The proof of Theorem~\ref{thm:LAP_NLM} will be given in Appendix~B. With these technical preparations we
 have all the tools to prove our main results in the following sections.
    
\section{Proof of Theorem~\ref{thm:existence_via_BanachFPT}} \label{sec:ProofThm1}

We prove Theorem~\ref{thm:existence_via_BanachFPT} with the aid of Banach's Fixed Point Theorem following the
approach by Guti\'{e}rrez~\cite{Gut_nontrivial}. We consider the map $T(\cdot,h):L^q(\R^n)\to L^q(\R^n)$
given by
\begin{align} \label{eq:defnT}
  T(u,h) :=   \widehat{h\,d\sigma_\lambda} + \fR_\lambda (f(\cdot,\chi(u)))
\end{align}
where $h\in X_\lambda^\delta$ as in Proposition~\ref{prop:Herglotz_waves} and $\chi$ is a smooth function such
that $|\chi(z)|\leq \min\{|z|,1\}$ and $\chi(z)=z$ for $|z|\leq \frac{1}{2}$. In view of the
properties of the Herglotz waves and $\fR_\lambda$ mentioned earlier,  a fixed point of $T(\cdot,h)$ is a
strong solution of the equation $-\Delta u -\lambda u = f(x,\chi(u))$ in $\R^n$. Since that fixed point $u$
will belong to a small ball in $L^q(\R^n)$, we will be able to show $\chi(u)=u$ so that a solution
of~\eqref{eq:NLH} is found. The choice of the exponent $q$ is delicate; it's the major technical issue in our
approach. Our assumption~\eqref{eq:choice_of_ps} implies that the set
\begin{equation*}
   \Xi_{s,p}:= \left\{q\in \left(\frac{2n}{n-1},\frac{2n}{(n-3)_+} \right) : \; q<
   \min\Big\{\frac{s(n+1)(p-2)}{(2s-(n+1))_+},  \frac{2ns(p-1)}{(s(n+1)-2n)_+}\Big\}
   \right\}   
\end{equation*}
is non-empty, so we may choose some arbitrary but fixed $q\in \Xi_{s,p}$ throughout this section.  
 
 \begin{prop}\label{prop:estimates_welldefinedness}
   Assume (A) and $\lambda>0,h\in X_\lambda^\delta$. Then  the map $T(\cdot,h):L^q(\R^n) \to L^q(\R^n)$
   from~\eqref{eq:defnT} is well-defined and we have  
   \begin{align}\label{eq:estimate_selfmap}
     \begin{aligned}
     \|T(u,h)\|_q &\leq C(\|u\|_q^\alpha  + \|h \|_{C^m}), \\
     \|T(u,h)-T(v,h)\|_q &\leq  C \|u-v\|_q  (\|u\|_q+\|v\|_q)^{\alpha-1}
   \end{aligned}
   \end{align} 
   for some $\alpha>1$ and all $u,v\in L^q(\R^n)$.
 \end{prop}
 \begin{proof}
    Using (A),  H\"older's inequality and $|\chi(u)|\leq \min\{|u|,1\}$  we get for all $u\in L^q(\R^n)$ 
   \begin{align} \label{eq:def_t*} 
     \begin{aligned}
     &\|f(\cdot,\chi(u))\|_t
     \leq \| |Q| |\chi(u)|^{p-1}\|_t
     \leq \|Q\|_{\tilde s} \|\chi(u)^{p-1}\|_{\frac{t\tilde s}{\tilde s-t}}  
     \leq \|Q \|_{\tilde s} \|u\|_q^{\frac{q}{t}-\frac{q}{\tilde s}}<\infty,\\
    &\text{provided}\quad  
    t^*(q,\tilde s):=\max\Big\{1,\frac{\tilde sq}{\tilde s(p-1)+q}\Big\}\leq t\leq \tilde s,\; \tilde s\in
    [s,\infty].
    \end{aligned} 
   \end{align} 
   In particular we find $f(\cdot,\chi(u))\in L^t(\R^n)$
   for all $t\in [t^*(q,s),\infty]$. For any such $t$ we choose $\tilde s:= \frac{tq}{(q-t(p-1))_+}\in
   [s,\infty]$ (largest possible) so that $t^*(q,\tilde s)\leq t\leq \tilde s$ holds. So the previous estimate
   gives for  $\alpha_t:= \frac{q}{t}-\frac{q}{\tilde s}=\min\{\frac{q}{t},p-1\}$ 
   \begin{align}\label{eq:estimate_f}
     \begin{aligned}
     \|f(\cdot,\chi(u))\|_t
     \leq C \|u\|_q^{\alpha_t}<\infty  \qquad\text{for all }t\in [t^*(q,s),\infty]. 
    \end{aligned} 
   \end{align}
   Now we have to choose $t\in [t^*(q,s),\infty]$ in such a way that the mapping properties of $\fR_\lambda$
   from Theorem~\ref{thm:LAP} ensure $\fR_\lambda(f(\cdot,\chi(u))) \in L^q(\R^n)$. In
   view of~\eqref{eq:admissible} we have to require 
   \begin{equation} \label{eq:admissible_t}
     \frac{nq}{n+2q}\leq t<\frac{(n+1)q}{n+1+2q}\qquad\text{and}\qquad t<\frac{2n}{n+1}.
   \end{equation}
   Since $q\in\Xi_{s,p}$ implies $q<\frac{2n}{(n-3)_+}$ and hence $\frac{nq}{n+2q}<\frac{2n}{n+1}$, 
   we can find such $t$  if and only if
   \begin{equation} \label{eq:tstarq}
     t^*(q,s)< \frac{(n+1)q}{n+1+2q}\qquad\text{and}\qquad t^*(q,s)<\frac{2n}{n+1}. 
   \end{equation}
   These two inequalities hold due to $q\in\Xi_{s,p}$. From this,
   Proposition~\ref{prop:Herglotz_waves} and Theorem~\ref{thm:LAP} we get  
   \begin{align*} 
     \|T(u,h)\|_q
     &\leq \|\fR_\lambda(f(\cdot,\chi(u)))\|_q +  \|\widehat{h \,d\sigma_\lambda}  \|_q \\
     &\leq C(\|f(\cdot,\chi(u))\|_t + \|h \|_{C^m} ) \\
     &\stackrel{\eqref{eq:estimate_f}}{\leq} C(\|u\|_q^{\alpha_t}  + \|h \|_{C^m}). 
   \end{align*}
   Since $t$ was chosen according to~\eqref{eq:admissible_t}, we have $t<q$ and thus
   $\alpha_t=\min\{p-1,\frac{q}{t}\}>1$.
   Moreover, from (A) and H\"older's inequality ($\frac{1}{t}=\frac{1}{\tilde
   s}+\frac{1}{q}+\frac{\alpha_t-1}{q}$) we get  
  \begin{align} \label{eq:contraction_estimate}
    \begin{aligned}
    \|f(\cdot,\chi(u))-f(\cdot,\chi(v)) \|_t  
    &\leq C\Big\||Q||\chi(u)-\chi(v)|(|\chi(u)|+|\chi(v)|)^{p-2}\Big\|_t   \\
    &\leq  C\|Q\|_{\tilde s} \|\chi(u)-\chi(v)\|_q \|
    (|\chi(u)|+|\chi(v)|)^{p-2}\|_{\frac{q}{\alpha_t-1}}   \\
    &\leq  C \|u-v\|_q \|(|u|+|v|)^{\alpha_t-1}\|_{\frac{q}{\alpha_t-1}}  \\
    &\leq   C \|u-v\|_q \big(  \|u\|_q+\|v\|_q\big)^{\alpha_t-1}.
  \end{aligned}
  \end{align}
  Here $p-2\geq \alpha_t-1>0$ was used. Hence we get 
  \begin{align*}
    \|T(u,h)-T(v,h)\|_q
    &\leq \|\fR_\lambda(f(\cdot,\chi(u)))-\fR_\lambda(f(\cdot,\chi(v)))\|_q \\ 
    &\leq C \|u-v\|_q  (\|u\|_q+\|v\|_q)^{\alpha_t-1},
  \end{align*}
  which finishes the proof.
 \end{proof}

  \medskip
  
  \noindent \textbf{Proof of Theorem~\ref{thm:existence_via_BanachFPT}:} 
  
  \medskip
  
  \noindent\textit{Step 1: Existence of a solution continuum $(u_h)$ in $L^q(\R^n)$:} 
  We apply Banach's Fixed Point Theorem to $T(\cdot,h)$ on a closed small ball around zero $\ov{B_\rho}\subset
  L^q(\R^n)$, $q\in\Xi_{s,p}\neq \emptyset$ and $h\in X^\delta$ with $\delta>0$ sufficiently small.
  From~\eqref{eq:estimate_selfmap} we get that $T(\cdot,h):\ov{B_\rho}\to \ov{B_\rho}$ is a contraction
  provided $\rho,\delta>0$ are chosen sufficiently small. It is even a
  uniform contraction  since its Lipschitz constant is  independent of $h$.   
  Moreover, $T$ is continuous with respect to the topology of $L^q(\R^n)\times
  C^m(S^{n-1}_\lambda;\C)$, which follows from Proposition~\ref{prop:Herglotz_waves} and
  Proposition~\ref{prop:estimates_welldefinedness}.
  Hence, Banach's Fixed Point Theorem for continuous uniform contractions yields a continuum of (uniquely determined) fixed points $u_h\in \ov{B_\rho}$ of
  $T(\cdot,h)$ for $h\in X^\delta$ provided $\delta>0$ is small enough.
  
  \medskip
  
  \noindent\textit{Step 2: The continuum property in $L^r(\R^n)$ for $r\in (q,\infty]$:}\; As a fixed point of
  $T(\cdot,h)$ the function $u_h$ solves $-\Delta u - \lambda u = f(x,\chi(u))$ in the strong sense on $\R^n$, see
  Theorem~\ref{thm:LAP}. Since $f(\cdot,\chi(u_h))$ is bounded, we even have $u_h\in W^{2,r}_{loc}(\R^n)$ for
  all $r\in [1,\infty)$. 
  Fixing now $\tilde q\geq q$ such that $\tilde q>\frac{n}{2}$ we get from Theorem~8.17 in \cite{GiTr}
  (Moser iteration) 
   \begin{align*}
    \|u_{h_1}-u_{h_2}\|_\infty
    &\leq \|u_{h_1}-u_{h_2}\|_q + \|f(\cdot,\chi(u_{h_1}))- f(\cdot,\chi(u_{h_2}))\|_{\tilde q} \\ 
    &\leq \|u_{h_1}-u_{h_2}\|_q + \|Q\|_\infty
    \|(|\chi(u_{h_1})|^{p-2}+|\chi(u_{h_2})|^{p-2})(\chi(u_{h_1})-\chi(u_{h_2}))\|_{\tilde q} \\
    &\leq \|u_{h_1}-u_{h_2}\|_q + 2^{p-2+\frac{\tilde q-q}{q}} \|Q\|_\infty
    \||u_{h_1}-u_{h_2}|^{q/\tilde q}\|_{\tilde q} \\
    &\leq C(\|u_{h_1}-u_{h_2}\|_q + \|u_{h_1}-u_{h_2}\|_q^{q/\tilde q})  
   \end{align*} 
   so that $(u_h)$ is also a continuum in $L^\infty(\R^n)$ and hence in $L^r(\R^n)$ for all $r\in (q,\infty]$.
   
   \medskip
   
   \noindent\textit{Step 3: The continuum property in $L^r(\R^n)$ for $r\in (\frac{2n}{n-1},q)$:}\; 
    From  $u_h\in L^q(\R^n)$ we deduce $f(\cdot,\chi(u_h))\in L^{t^*(q,s)}(\R^n)\cap L^\infty(\R^n)$ for
    $t^*(q,s)$ defined in~\eqref{eq:def_t*}. We set 
    $$
     \tilde q:= \max\Big\{r,\frac{(n+1)t^*(q,s)}{n+1-2t^*(q,s)}\Big\} 
    $$
    so that Theorem~\ref{thm:LAP} gives $u_h\in L^{\tilde q}(\R^n)$ since the tuple of exponents
    $(t^*(q,s),\tilde q)$  satisfies the inequalities~\eqref{eq:admissible}.  Notice that $q <
    \frac{s(n+1)(p-2)}{(2s-(n+1))_+}$ (because of $q\in\Xi_{s,p}$) implies $\tilde q<q$. 
     Moreover,  we have 
    \begin{align*}
     \|u_{h_1}-u_{h_2}\|_{\tilde q}
     &\leq \|\mathcal F( (h_1-h_2)\,d\sigma_\lambda)\|_{\tilde q} +  \|\mathfrak
     R_\lambda(f(\cdot,\chi(u_{h_1}))- f(\cdot,\chi(u_{h_2})))\|_{\tilde q} \\
     &\leq C \|h_1-h_2\|_{C^m}  + C \|f(\cdot,\chi(u_{h_1}))- f(\cdot,\chi(u_{h_2})))\|_{t^*(q,s)} \\
     &\stackrel{\eqref{eq:contraction_estimate}}{\leq} C (\|h_1-h_2\|_{C^m} + 
     \|u_{h_1}-u_{h_2}\|_q \big(\|u_{h_1}\|_q+\|u_{h_2}\|_q\big)^{p-2}) \\ 
     &\leq C( \|h_1-h_2\|_{C^m} + \| u_{h_1}-u_{h_2}\|_q).
    \end{align*}
    Taking now $\tilde q$ as the new $q$ and repeating the above arguments we get after finitely many
    steps $u_h\in L^r(\R^n)$ as well as 
    $$
     \|u_{h_1}-u_{h_2}\|_r \leq C (\|h_1-h_2\|_{C^m} +\|u_{h_1}-u_{h_2}\|_q). 
    $$ 
   Hence, $(u_h)$ is a continuum in $L^r(\R^n)$ for all $r\in (\frac{2n}{n-1},q)$.
   
   \medskip
   
   \noindent\textit{Step 4: The continuum property in $W^{2,r}(\R^n)$ for $r\in (\frac{2n}{n-1},\infty)$ and
   \eqref{eq:NLH}:}\; From step 2 and step 3 we get 
   $$
     -\Delta u_h + u_h = (1+\lambda)u_h+f(\cdot,\chi(u_h))\in L^r(\R^n)
   $$ 
   because of $|f(x,\chi(u_h))|\leq \|Q\|_\infty |u_h|\in L^r(\R^n)$. Bessel potential estimates imply
   $u_h\in W^{2,r}(\R^n)$ and as above one finds $\|u_{h_1}-u_{h_2}\|_{W^{2,r}(\R^n)} \leq C \|u_{h_1}-u_{h_2}\|_r$  so that the
   continuum property is proved. In particular,   $\|h\|_{C^m}\to 0$ implies
   $\|u_h\|_\infty=\|u_h-u_0\|_\infty \to 0$ and thus $\chi(u_h)=u_h$ is a $W^{2,r}(\R^n)$-solution
   of~\eqref{eq:NLH} for all $h\in X^\delta_\lambda$ provided $\delta>0$ is sufficiently small.   
  
 \medskip
  
  \noindent \textit{Step 5: Asymptotics of $u_h$:} From the previous steps we get 
  $u_h\in L^r(\R^n)$ for all $r\in (\frac{2n}{n-1},\infty]$ and this implies
  $|f(\cdot,u_h)|\leq Q|\chi(u_h)|^{p-1} \in L^t(\R^n)$ for $t> \frac{2ns}{2n+s(n-1)(p-1)}$. 
  In particular we have $f(\cdot,u_h)\in L^{\frac{2(n+1)}{n+3}}(\R^n)$ because
  $$
    p \stackrel{\eqref{eq:choice_of_ps}}{>} 
    \frac{2s(n^2+2n-1)-2n(n+1)}{(n^2-1)s}
    > \frac{s(2n^2+3n-1)-2n(n+1)}{(n^2-1)s}.
  $$
  Hence, Proposition~\ref{prop:LAP_asymptotics} yields   
  $$ 
      \lim_{R\to\infty} \frac{1}{R} \int_{B_R} \left| \fR_\lambda(f(\cdot,u_h))(x) - 
     \sqrt{\frac{\pi}{2\lambda}}\left(\frac{\sqrt\lambda}{|x|}\right)^{\frac{n-1}{2}}\Real\left(
     e^{i(\frac{n-3}{4}\pi-\sqrt\lambda|x|)}
     \widehat{f(\cdot,u_h)}(-\sqrt\lambda \hat x) \right)\right|^2\,dx
     = 0.
   $$
   Using $\ov{h(-\xi)}=h(\xi)$ and  Proposition~\ref{prop:Herglotz_waves} we moreover get
   \begin{align*}
     \lim_{R\to\infty} \frac{1}{R}\int_{B_R} \left|\widehat{h\,d\sigma_\lambda}(x) -  
     \sqrt{\frac{2}{\pi}} \left(\frac{\sqrt\lambda}{|x|}\right)^{\frac{n-1}{2}}
     \Real\left(e^{i(\frac{n-1}{4}\pi-\sqrt\lambda|x|)} h(\sqrt\lambda \hat x)\right)
     \right|^2\,dx = 0.
   \end{align*}
   So $u_h = T(u_h,h) = \widehat{h\,d\sigma_\lambda} + \fR_\lambda(f(\cdot,u_h))$ 
   and the above asymptotics imply~\eqref{eq:farfield}. Finally, in the case $p>\frac{s(3n-1)-2n}{(n-1)s}$ we
   have $V:=Q|u|^{p-2}\in L^t(\R^n)$ for some $t<\frac{2n}{n+1}$ because of $Q\in L^s(\R^n)$ and $u\in L^r(\R^n)$ for all
   $r>\frac{2n}{n-1}$. Hence, Lemma~2.9 in~\cite{EvWe_dual} yields the pointwise bounds if $n\geq 3$ and 
   Lemma~2.3 in~\cite{Eveq_plane} in the case $n=2$.
   
   \medskip
   
   \noindent \textit{Step 6: Distinguishing $u_{h_1},u_{h_2}$:} From Step 2 we deduce that $h_1\neq h_2$
   implies $u_{h_1}\neq u_{h_2}$. Indeed, assuming  $u_{h_1}=u_{h_2}$ we get  
   $T(u_{h_1},h_1)=T(u_{h_2},h_2)$ and thus $\mathcal F\big( (h_1-h_2)\,d\sigma_\lambda\big)=0$. We show that
   this implies $h_1=h_2$. To see this we use the scaling  property 
   $$
    \mathcal F(h\,d\sigma_\lambda)(x) =  \lambda^{\frac{n-1}{2}} 
    \mathcal F\big(h(\sqrt\lambda\cdot)\,d\sigma_1\big)(\sqrt\lambda x). 
  $$
  From Corollary~4.6 in~\cite{Agmon_Arepresentation} we infer
   \begin{align*}
     0
     &= \lim_{R\to\infty} \frac{1}{R} \int_{B_R} |\mathcal F\big((h_1-h_2)\,d\sigma_\lambda\big)|^2 \,dx \\
     &=  \lambda^{n-1}  \lim_{R\to\infty} \frac{1}{R} \int_{B_R}
     |\mathcal F\big((h_1-h_2)(\sqrt\lambda\cdot)\,d\sigma_1\big)(\sqrt\lambda x)|^2 \,dx \\
     &= \lambda^{\frac{n-1}{2}}\cdot \lim_{R\to\infty} \frac{1}{\sqrt\lambda R} \int_{B_{\sqrt\lambda R}}
     |\mathcal F\big((h_1-h_2)(\sqrt\lambda\cdot)\,d\sigma_1\big)(x)|^2 \,dx \\
     &= 2(2\pi)^{n-1} \lambda^{\frac{n-1}{2}} \int_{S^{n-1}} |(h_1-h_2)(\sqrt\lambda x)|^2\,d\sigma_1(x), 
   \end{align*}
   which implies $h_1=h_2$. \qed

  \medskip
  
  \begin{bem} \label{Remark2}
  \begin{itemize}
%
  \item[(a)] Let us describe how Theorem~\ref{thm:existence_via_BanachFPT} provides nonsymmetric solutions of
  symmetric nonlinear Helmholtz equations as mentioned in Remark~\ref{Remark1}(b). We assume $f(\gamma
  x,z)=f(x,z)$ for almost all $x\in\R^n$ and all $z\in\R,\gamma\in\Gamma$ where $\Gamma\subset O(n),\Gamma\neq \{\id\}$ is a subgroup. Since $\Gamma\neq
  \{\id\}$ we can find $h\in X^\delta$ and $\gamma\in\Gamma$ satisfying $h\neq h\circ\gamma$ and our claim is
  that this implies $u_h\neq u_h\circ\gamma$. Indeed, otherwise we would have
  \begin{align*}
    \widehat{h\,d\sigma_\lambda} + \fR_\lambda(f(\cdot,\chi(u_h)))  
    &=T(u_h,h) = u_h = u_h\circ\gamma = T(u_h,h)\circ\gamma \\ 
    &= \widehat{h\,d\sigma_\lambda}\circ \gamma + \fR_\lambda(f(\cdot,\chi(u_h))) \circ\gamma \\
    &= \widehat{h\circ \gamma\,d\sigma_\lambda} + \fR_\lambda(f(\cdot,\chi(u_h)) \circ\gamma) \\
    &= \widehat{h\circ \gamma\,d\sigma_\lambda} + \fR_\lambda(f(\cdot,\chi(u_h\circ\gamma)) ) \\
    &= \widehat{h\circ \gamma\,d\sigma_\lambda} + \fR_\lambda(f(\cdot,\chi(u_h))).
  \end{align*}
  From the second to the third line we used that $\fR_\lambda$ is a convolution operator with a radially
  symmetric and hence $\Gamma$-symmetric kernel and from the third to the fourth line we used that $f$ is
  $\Gamma$-invariant. So we conclude $\widehat{h\,d\sigma_\lambda} =
  \widehat{h\circ\gamma\,d\sigma_\lambda}$, which implies $h=h\circ\gamma$ as in Step 6 above, a
  contradiction. Hence, $u_h\neq u_h\circ\gamma$ so that $u_h$ is not $\Gamma$-symmetric.
   \item[(b)] In Remark~\ref{Remark1}(c) we claimed that Theorem~\ref{thm:existence_via_BanachFPT}
    provides radial solutions assuming the weaker condition~\eqref{eq:choicep_radial} instead of
    \eqref{eq:choice_of_ps}. This is due to an improved version of the resolvent estimates
    from~\eqref{eq:admissible} where in both lines $\frac{2}{n+1}\leq \frac{1}{t}-\frac{1}{q}$ can be
    replaced by $\frac{3n-1}{2n^2} <\frac{1}{t}-\frac{1}{q}$. This was demonstrated in Remark~3.1
    in~\cite{BoCaMa_4thorder}.
    Let us explain how these improved resolvent estimates allow to obtain radial solutions for a larger range
    of exponents.
    The only radially symmetric Herglotz waves $\widehat{h\,d\sigma_\lambda}$ are given by real-valued
    and constant densities $h$. So for $h\in\R$ we get that $T(\cdot,h)$ maps $L^q_{rad}(\R^n)$ into itself
    provided $f(x,u)=f_0(|x|,u)$ is radially symmetric. Here, the exponent $q$ may be chosen from 
    $$ 
      \Xi_{s,p}^{rad}:= \left\{q\in
      \left(\frac{2n}{n-1},\frac{2n}{(n-3)_+} \right) : \; q< \min\Big\{\frac{2n^2s(p-2)}{((3n-1)s-2n^2)_+}, 
      \frac{2ns(p-1)}{(s(n+1)-2n)_+}\Big\} \right\},  
    $$
    which now is nonempty due to~\eqref{eq:choicep_radial}. Replacing in
    Proposition~\ref{prop:estimates_welldefinedness} the first inequality in~\eqref{eq:tstarq} by 
    $t^*(q,s)<\frac{2n^2q}{2n^2+(3n-1)q}$ and redefining $\tilde q$ in the proof of Theorem~\ref{thm:existence_via_BanachFPT}
    accordingly, we get the desired existence result again from Banach's Fixed Point Theorem. 
  \item[(c)] Under severe restrictions on the nonlinearity our result
  may also be proved using dual variational methods originally developed by Ev\'{e}quoz and Weth
  \cite{EvWe_dual}.
  To demonstrate this we consider the special case $f(x,z)=|z|^{p-2}z$ with
  $\frac{2(n+1)}{n-1}<p<\frac{2n}{n-2}$.
  The dual functional $J_h:L^{p'}(\R^n)\to\R$ is then given by 
  $$ 
    J_h(v) := \frac{1}{p'}\int_{\R^n} |v|^{p'} - \int_{\R^n} v\cdot \mathcal F(h\,d\sigma_\lambda) 
      - \frac{1}{2} \pv \int_{\R^n} \frac{|\hat v(\xi)|^2}{|\xi|^2-\lambda}\,d\xi  
  $$
  and a local minimizer of $J_h$ lying in the interior of a small ball  may be shown
  to exist for $h\in X^\delta$ for $\delta>0$ sufficiently small using Ekeland's variational
  principle. Since every critical point $v_h$ of $J_h$ provides a fixed point  of $T(\cdot,h)$ vai
  $u_h:=|v_h|^{p'-2}v_h$, see Section~4 in~\cite{EvWe_dual}, we rediscover the solutions found in
  Theorem~\ref{thm:existence_via_BanachFPT}.
\end{itemize}
\end{bem}

 \section{Proof of Theorem~\ref{thm:4thorder}} \label{sec:ProofThm2}
 
 In this section we discuss how the above approach  
 needs to be modified in order to get solutions of the fourth order problem~\eqref{eq:4thorderNLH}.  

 \medskip
 
 We first consider the case (i) in~\eqref{eq:case_distinction}.
 From Theorem~\ref{thm:4thorder_resolvent} we know that there is a  resolvent-type operator
 $\bold{R}$ associated with $\Delta^2-\beta \Delta + \alpha$ which is linear and bounded between
 the same (and even more) pairs of Lebesgue spaces as $\fR_\lambda$. So all estimates in
 Proposition~\ref{prop:estimates_welldefinedness} involving $\fR_{\lambda}$ equally hold  for $\bold{R}$.
 As a replacement for the Herglotz wave of the Helmholtz operator we take again a Herglotz wave
 $\widehat{h\,d\sigma_\lambda}$ where now $h\in Y^\delta=X_\lambda^\delta$ and $\lambda$ was defined
 in~\eqref{eq:Ydelta} in dependence of $\alpha,\beta$. This definition of $\lambda$ was made in such a way
 that $\widehat{h\,d\sigma_\lambda}$ satisfies the homogeneous equation
 $\Delta^2\phi-\beta\Delta\phi+\alpha\phi=0$ because of $\lambda^2-\beta \lambda+\alpha=0$.
 As a consequence, also the Herglotz-wave part of the map 
 $$
   T(u,h) := \widehat{h\,d\sigma_\lambda} + \bold{R}(f(\cdot,\chi(u))),
 $$
 may be estimated as in Proposition~\ref{prop:estimates_welldefinedness}. So we can find  
 a fixed point of $T(\cdot,h)$ in a small ball of $L^q(\R^n)$ for $q\in \Xi_{s,p}$ exactly as in
 the proof of Theorem~\ref{thm:existence_via_BanachFPT}. The qualitative properties can also be proved the same way,
 see also Section~5 in~\cite{BoCaMa_4thorder} where $u_h\in W^{4,r}(\R^n)$ for all
 $r\in(\frac{2n}{n-1},\infty)$ as well as its pointwise decay rate was proved in the special case
 $f(x,z)=\Gamma(x)|z|^{p-2}z,\,\Gamma\in L^\infty(\R^n)$.
 
 \medskip
 
 In the case (ii) the proof is essentially the same. The only difference is that the fixed point operator now
 reads 
 $$
   T(u,h) := \widehat{h_1\,d\sigma_{\lambda_1}}+\widehat{h_2\,d\sigma_{\lambda_2}} +
   \bold{R}(f(\cdot,\chi(u))), 
 $$
 for $h=(h_1,h_2)\in Y^\delta = X_{\lambda_1}^\delta\times X_{\lambda_2}^\delta$. Besides that all arguments
 are identical and we conclude as above.~\qed
 
   
\section{Proof of Theorem~\ref{thm:NLM}} \label{sec:ProofThm3}

Theorem~\ref{thm:NLM} will be proved via the Contraction Mapping
Theorem on a small ball in $\R^3$ (part (i)) or on $\R^3$ (part (ii)). 
The reason for this is that the Limiting Absorption Principle in the latter
case is much weaker and forces us to consider nonlinear curl-curl equations with
nonlinearities that grow sublinearly at infinity.
Notice that the growth of the nonlinearity with respect to $E$ cannot be ignored by using a truncation  as
in our results proved above. In fact, an equivalent of local elliptic regularity theory and
in particular $(L^r,L^\infty)$-estimates for the curl-curl-operator are not known and may even be false.
 
\medskip

We start with a few words on the proof of part (i), which is very similar to the proof of
Theorem~\ref{thm:existence_via_BanachFPT} and Theorem~\ref{thm:4thorder}. So let $f$ satisfy (A'). For $q\in
\Xi_{s,p}$ defined in~\eqref{eq:choice_of_ps} (for $n=3$) we consider the map $T(\cdot,h):L^q_{cyl}(\R^3;\R^3)\to
L^q_{cyl}(\R^3;\R^3)$ where
\begin{align} \label{eq:defnT_NLM} 
  T(E,h) :=   \widehat{h\,d\sigma_\lambda} + R_\lambda (f(\cdot,\chi(|E|)E/|E|)).
\end{align}
Here, $L^q_{cyl}(\R^3;\R^3)$ denotes the Banach space of cylindrically symmetric functions
lying in $L^q(\R^3;\R^3)$ and the function $\chi\in C^\infty(\R)$ is chosen as before, i.e.,  it
satisfies $|\chi(z)|\leq \min\{|z|,1\}$ as well as $\chi(z)=z$ provided $|z|\leq \frac{1}{2}$.  
The map $T(\cdot,h)$ is well-defined for  $h\in Z_{cyl}^\delta, \delta>0$ since the
nonlinearity is compatible with cylindrical symmetry by (A').
By Proposition~\ref{prop:Herglotz_waves_NLM} the functions $\widehat{h\,d\sigma_\lambda}$ are smooth
cylindrically symmetric solutions of $\nabla\times\nabla\times E-\lambda E=0$ and satisfy the same estimates
as their scalar counterparts used in the proof of Theorem~\ref{thm:existence_via_BanachFPT}.
Likewise, Theorem~\ref{thm:LAP_NLM} implies that $R_\lambda$ restricted to the space of cylindrically
symmetric functions has the same $L^p-L^q$ mapping properties as  $\fR_\lambda$. Moreover, 
not only $f$ but also the function $(x,E)\mapsto f(x,\chi(|E|)E/|E|)$ satisfies (A') because
$ z\mapsto \chi(z)/z $ is smooth. So the operator $T$ defined in~\eqref{eq:defnT_NLM} also
satisfies the estimates from Proposition~\ref{prop:estimates_welldefinedness} and one obtains a unique fixed
point $E_h$ of $T(\cdot,h)$ on a small ball in $L^q_{cyl}(\R^3;\R^3)$ via the Contraction Mapping Theorem. 
As a cylindrically symmetric and hence divergence-free solution of~\eqref{eq:NLM} the vector field $E_h$ even
solves the elliptic system~\eqref{eq:NLM_cylindrical} so that
elliptic regularity theory implies $\chi(E_h)=E_h$ provided the ball in $L^q_{cyl}(\R^3;\R^3)$ and
$h\in Z_{cyl}^\delta, \delta>0$ are chosen small enough. Also the estimate $|E_h(x)|\leq
C_h(1+|x|)^{\frac{1-n}{2}}$ is proved as in step 5 of the proof of Theorem~\ref{thm:existence_via_BanachFPT}  under the assumption
$p>\frac{s(3n-1)-2n}{(n-1)s}$.  
 
\medskip 

From now on we prove part (ii), so let $f$ satisfy assumption (B). We fix an exponent $q$ such that  
$3<q<\frac{3s}{(2s-3)_+}$,  which is possible due to $s\in [1,2]$. For $h\in Z$ we set 
\begin{align} \label{eq:defnT_NLM2}
  T(E,h) :=   \widehat{h\,d\sigma_\lambda} + R_\lambda (f(\cdot,E)).
\end{align}
We first verify that $T(\cdot,h):L^q(\R^3;\R^3)\to L^q(\R^3;\R^3)$ is well-defined and Lipschitz continuous. 
In the proof of these estimates we use the number
$$
  \alpha_{p,\tilde p} :=  \sup_{z\in\R} |z|^{p-2}(1+|z|)^{\tilde p-p} = \frac{(p-2)^{p-2}(2-\tilde p)^{2-\tilde
  p}}{(p-\tilde p)^{p-\tilde p}} 
  \qquad\text{where }0^0:=1 \text{ and }\tilde p\leq 2\leq p. 
$$
 
 \begin{prop}\label{prop:estimates_welldefinedness_NLM}
    Assume (B) and $\lambda>0,h\in Z$. Then the map $T(\cdot,h):L^q(\R^3;\R^3) \to L^q(\R^3;\R^3)$
   from~\eqref{eq:defnT_NLM2} is well-defined with  
   \begin{align}\label{eq:estimate_selfmap_NLM}
     \begin{aligned}
      \|T(E_1,h)-T(E_2,h)\|_q \leq C\alpha_{p,\tilde p}(\|Q\|_s+\|Q\|_\infty)\|E_1-E_2\|_q. 
   \end{aligned}
   \end{align}
   where $C$ only depends on $q$ and $s$.
 \end{prop}
 \begin{proof}
   By Proposition~\ref{prop:Herglotz_waves_NLM} the functions $\widehat{h\,d\sigma_\lambda}$ belong to
   $L^q(\R^3;\R^3)$ for all $h\in Z$. So the definition of $T$ from~\eqref{eq:defnT_NLM2} and the Limiting
   Absorption Principle for the curl-curl operator (Theorem~\ref{thm:LAP_NLM}) imply that
   $T(\cdot,h)$ is well-defined if we can show $f(\cdot,E)\in L^q(\R^3;\R^3)\cap L^t(\R^3;\R^3)$ for some
   $t\in (1,\infty)$ satisfying~\eqref{eq:admissible}. To verify this we set
   $\tilde s:=\max\{s,\frac{3}{2}\}$ for $s\in [1,2]$ as in assumption (B) and choose $t:=\frac{\tilde
   sq}{\tilde s+q}$. This implies $1<t<\frac{3}{2}, \frac{3}{2}\leq  \tilde s \leq  2$  so that $(t,q)$
   indeed satisfies~\eqref{eq:admissible}. 
   So we infer from assumption (B)
   \begin{align*}
     \|f(\cdot,E)\|_q
     &\leq \|Q  |E|^{p-1}(1+|E|)^{\tilde p-p}\|_q
     \leq \alpha_{p,\tilde p} \|Q  |E|\|_q   
     \leq  \alpha_{p,\tilde p} \|Q\|_\infty  \| E\|_q
     < \infty, \\
     \|f(\cdot,E)\|_t
     &\leq \|Q  |E|^{p-1}(1+|E|)^{\tilde p-p}\|_t
     \leq \alpha_{p,\tilde p} \|Q  |E|\|_t   
     \leq  \alpha_{p,\tilde p} \|Q\|_{\tilde s}  \|E\|_q   
     < \infty.  
   \end{align*}
   This implies that $T(\cdot,h)$ is well-defined. Moreover, we have
   \begin{align*} 
     \|f(\cdot,E_1)-f(\cdot,E_2)\|_q
      &\leq  \alpha_{p,\tilde p} \|Q\|_\infty   \|E_1-E_2\|_q,  \\
     \|f(\cdot,E_1)-f(\cdot,E_2)\|_t
      &\leq  \alpha_{p,\tilde p} \|Q\|_{\tilde s}   \|E_1-E_2\|_q    
      \leq   \alpha_{p,\tilde p}(\|Q\|_s+\|Q\|_\infty)  \| E_1-E_2\|_q. 
   \end{align*}
   Combining these estimates with Theorem~\ref{thm:LAP_NLM} one gets \eqref{eq:estimate_selfmap_NLM} from the
   definition of $T(\cdot,h)$.
%
%
 \end{proof} 
 
 \medskip
   
 \noindent
 \textbf{Proof of Theorem~\ref{thm:NLM}~(ii):}\;
 From the previous proposition we get that $T(\cdot,h)$ maps $L^q(\R^3;\R^3)$ into itself and it is a
 contraction provided $\alpha_{p,\tilde p}(\|Q\|_s+\|Q\|_\infty)$ is small enough, which is guaranteed by the
 assumptions of the Theorem~\ref{thm:NLM}.
  So for any given $h\in Z$ the Contraction Mapping Theorem and Theorem~\ref{thm:LAP_NLM} provide a unique
  weak solution $E_h\in H_{loc}(\curl;\R^3)\cap L^q(\R^3;\R^3)$ of~\eqref{eq:NLH}. It remains to discuss the integrability properties
 of $E_h$. In this discussion we will w.l.o.g. assume that assumption (B) holds with $\tilde p\in (1,2]$
 because otherwise we may simply increase $\tilde p$.
 
 \medskip
%
 
 \noindent\textit{Proof of (ii)(a):}\;
 Under the additional assumption $(p,s)\neq (2,2)$ we want to show $E_h\in L^r(\R^3;\R^3)$ for all $r\in
 (3,q)$. To achieve this iteratively we use $E_h\in L^q(\R^3;\R^3)$ and hence, by assumption (B),
 $$
   f(\cdot,E_h)\in L^t(\R^3;\R^3)\cap L^{\tilde q}(\R^3;\R^3)
   \qquad\text{for all }t,\tilde q\in \left[t^*(q,s),\frac{q}{\tilde p-1}\right].
 $$ 
 This follows as in~\eqref{eq:def_t*}, where also $t^*(q,s)$ is defined. 
 In order to prove $E_h\in L^{\tilde q}(\R^3;\R^3)$ with $\tilde q<q$ we use $E_h=T(E_h,h)=
 \widehat{h\,d\sigma_\lambda} + R_\lambda (f(\cdot,E_h))$. In view of
 Proposition~\ref{prop:Herglotz_waves_NLM} and the mapping properties of $R_\lambda$ from
 Theorem~\ref{thm:LAP_NLM} we obtain $E_h\in L^{\tilde q}(\R^3;\R^3)$ with $\tilde q<q$ provided the pair 
 $(t,\tilde q)$ satisfies $1<t,\tilde q<\infty$ as well as the inequalities
 from~\eqref{eq:admissible} in the three-dimensional case $n=3$. In other words we require  
 \begin{equation} \label{eq:NLM_iteration_condition_tq}
   1<t<\frac{3}{2},\qquad 3<\tilde q<q,\qquad \frac{1}{2}\leq \frac{1}{t}-\frac{1}{\tilde q}\leq
   \frac{2}{3},\qquad t^*(q,s)\leq t,\tilde q \leq \frac{q}{\tilde p-1}.
 \end{equation}
 So we set $\tilde q:=\max\{r,\frac{2t}{2-t},t^*(q,s)\}=\max\{r,\frac{2t}{2-t}\}\geq r>3$. (Notice that our
 choice for $t$ will ensure $t^*(q,s)<t<\frac{3}{2}<3<r$.) Plugging in the definition of $t^*(q,s)$
 from~\eqref{eq:def_t*} we obtain after some calculations that the  
 inequalities~\eqref{eq:NLM_iteration_condition_tq} hold if
 \begin{equation*} 
   \max\Big\{1, \frac{3r}{3+2r}, \frac{qs}{q+s(p-1)} \Big\} < t < \min\Big\{ \frac{3}{2}, \frac{2q}{q+2},
   \frac{q}{\tilde p-1}\Big\}.
 \end{equation*}
 Here non-strict inequalities in~\eqref{eq:NLM_iteration_condition_tq} were sharpened
 to strict inequalities for notational convenience. Such a choice for $t$ is possible 
 because of $1\leq s\leq 2\leq p, (p,s)\neq (2,2)$ and $3<r<q< <
 \frac{3s}{(2s-3)_+}<\frac{3s(p-1)}{(2s-3)_+}$.
 For instance we may choose 
 $$ 
   t=t_q := \frac{1}{2}\cdot \left(  \max\Big\{1, \frac{3r}{3+2r}, \frac{qs}{q+s(p-1)} \Big\} + 
   \min\Big\{ \frac{3}{2}, \frac{2q}{q+2}, \frac{q}{\tilde p-1}\Big\} \right) 
 $$
 and Theorem~\ref{thm:LAP_NLM} implies $E_h \in  L^{\tilde q}(\R^3;\R^3)$. In the case $\tilde q=r$
 we are done. Otherwise, we may repeat this argument replacing $q$ by $\tilde q$ so that $r\in (3,\tilde q)$.
 The corresponding iteration yields $E_h\in L^r(\R^3;\R^3)$ after finitely many steps.
 
 \medskip
 
 \noindent\textit{Proof of (ii)(b):}\; Using $\tilde p<2<p$ we now prove $E_h\in L^r(\R^3;\R^3)$ for
 all $r\in (q,\frac{3s(p-1)}{(2s-3)_+})$. In view of (B) and $E_h\in L^q(\R^3;\R^3)$ we now have to choose
 $t,\tilde q \in [t^*(q,s),\frac{q}{\tilde p-1}]$ with $\tilde q>q$ such that the pair $(t,\tilde q)$
 satisfies $1<t,\tilde q<\infty$ as well as 
 $$
   1<t<\frac{3}{2},\qquad q<\tilde q<\infty,\qquad \frac{1}{2}\leq \frac{1}{t}-\frac{1}{\tilde q}\leq
   \frac{2}{3},\qquad t^*(q,s)\leq t,\tilde q \leq \frac{q}{\tilde p-1}.
 $$
 So we set $\tilde q:= \min\{\frac{3t}{3-2t},\frac{q}{\tilde p-1}\}$ and due to $1<\tilde p<2$ it remains to
 choose $t$ such that
  \begin{equation*} 
   \max\Big\{1, \frac{3q}{3+2q}, \frac{qs}{q+s(p-1)} \Big\} < t < \min\Big\{ \frac{3}{2},
   \frac{2q}{q+2(\tilde p-1)}, \frac{q}{\tilde p-1}\Big\}.
 \end{equation*}
 Such a choice is possible thanks to $q<\frac{3s(p-1)}{(2s-3)_+}$ and $1\leq s\leq 2,1<\tilde p<2$.
 For instance we may choose 
 $$
   t=t_q := \frac{1}{2}\cdot \left(  
   \max\Big\{1, \frac{3q}{3+2q}, \frac{qs}{q+s(p-1)} \Big\}  + \min\Big\{ \frac{3}{2},
   \frac{2q}{q+2(\tilde p-1)}, \frac{q}{\tilde p-1}\Big\}
   \right). 
 $$
 Then Theorem~\ref{thm:LAP_NLM} implies $E_h \in  L^{\tilde q}(\R^3;\R^3)$ and we have $\tilde q>q$. We may
 repeat this argument as long as  $\tilde q<\frac{3s(p-1)}{(2s-3)_+}$ and thereby obtain $u\in
 L^r(\R^3;\R^3)$ for all  $r\in (q,\frac{3s(p-1)}{(2s-3)_+})$, which finishes the proof.~\qed

\section{Appendix A: Proof of Proposition~\ref{prop:Herglotz_waves}}

  In this section we give the proof of Proposition~\ref{prop:Herglotz_waves}  which we repeat for
  convenience. 
  
   \begin{prop*} 
   Let $n\in\N,n\geq 2$ and $m:= \lfloor \frac{n-1}{2}\rfloor+1$. Then for all $h\in
   C^m(S_\lambda^{n-1};\C)$ the Herglotz wave $\widehat{h\,d\sigma_\lambda}$ is an analytic solution of
   $-\Delta\phi-\lambda\phi=0$ in $\R^n$ and satisfies the estimate $|(\widehat{h\,d\sigma_\lambda)}(x)|\leq
   C\|h\|_{C^m}(1+|x|)^{\frac{1-n}{2}}$ as well as 
   \begin{equation*} 
     \lim_{R\to\infty} \frac{1}{R} \int_{B_R} \left| \widehat{h\,d\sigma_\lambda}(x) - 
      \frac{1}{\sqrt{2\pi}} \left(\frac{\sqrt\lambda}{|x|}\right)^{\frac{n-1}{2}}      
      m_h(x)  \right|^2\,dx
       = 0.
   \end{equation*}
   In particular, we have $\|\widehat{h\,d\sigma_\lambda}\|_r\leq C_r \|h\|_{C^m}$ for all $r>\frac{2n}{n-1}$. 
 \end{prop*}
  
  The asymptotics of the functions $\widehat{h\,d\sigma_\lambda}$ are proved using
  the method of stationary phase, but typically it is assumed that $h$ is smooth, see for
  instance Proposition 4,5,6 in Chapter VIII\S2 of~\cite{Ste_harmonic_analysis} or page~6-7 in
  \cite{Gut_nontrivial}. In spirit, the above result is not new, but we could not find a reference for it
  covering all space dimensions $n\geq 2$ with an explicit value for $m$. For this reason, we decided to
  present a proof here.
  
  \medskip\medskip
 
  \noindent\textbf{Proof of Proposition~\ref{prop:Herglotz_waves}:}\;
  We consider the Herglotz wave 
  $$
    \widehat{h\,d\sigma_\lambda}(x) = \frac{1}{(2\pi)^{n/2}} \int_{S_\lambda^{n-1}}
    h(\xi)e^{-i\skp{x}{\xi}}\,d\sigma_\lambda(\xi). 
  $$
  To investigate its pointwise behaviour as $|x|\to\infty$ let $Q=Q_x\in O(n)$ satisfy $Q^Tx=|x|e_n$, so  
  $$
    \widehat{h\,d\sigma_\lambda}(x) 
    = \frac{1}{(2\pi)^{n/2}} \int_{S_\lambda^{n-1}}  h(Q\xi)e^{-i|x|\xi_n}\,d\sigma_\lambda(\xi). 
  $$
  Now we choose $\eta_1,\ldots\eta_{2n} \in C^\infty(\R^n)$ such that $\eta_1+\ldots+\eta_{2n}=1$ on
  $S_\lambda^{n-1}$ and, for $k=1,\ldots,n$,
  \begin{align*} 
    \supp(\eta_{2k-1}) &\subset \{\xi=(\xi_1,\ldots,\xi_n)\in \R^n: \xi_k>+\sqrt\lambda\delta\},\\
    \supp(\eta_{2k}) &\subset \{\xi=(\xi_1,\ldots,\xi_n)\in \R^n: \xi_k< -\sqrt\lambda\delta\},
  \end{align*} 
  where $\delta\in (0,\frac{1}{\sqrt n})$ is fixed.  Notice that such a partition of unity exists since the
  open sets on the right hand side cover the sphere $S_\lambda^{n-1}$. We define
  $h_j(\xi):=(2\pi)^{-n/2} h(Q\xi)\eta_j(\xi)$ so that we have to investigate the integrals
  $$
    I_j := \int_{S_\lambda^{n-1}} h_j(\xi)e^{-i|x|\xi_n}\,d\sigma_\lambda(\xi)
    \qquad (j=1,\ldots,2n).
  $$ 
  
  \medskip
  
  We first estimate the oscillatory integrals $I_1,\ldots,I_{2n-2}$ where the resonant poles $\pm e_n$ are cut
  out by the choice of $\eta_1,\ldots,\eta_{2n-2}$. To estimate $I_j$ we use the local parametrization given by 
  \begin{align*}
    \text{If } j=2k-1 &:\quad \psi_j(\xi'):=\sqrt\lambda
    (\xi_1,\ldots,\xi_{k-1},+\sqrt{1-|\xi'|^2},\xi_k,\ldots,\xi_{n-1})  \;\in S^{n-1}_\lambda  \\
    \text{If } j=2k\quad\;\,  &:\quad \psi_j(\xi'):=\sqrt\lambda
    (\xi_1,\ldots,\xi_{k-1},-\sqrt{1-|\xi'|^2},\xi_k,\ldots,\xi_{n-1}) \;\in S^{n-1}_\lambda 
  \end{align*} 
  for $\xi':=(\xi_1,\ldots,\xi_{n-1})$ belonging to the unit ball $B \subset\R^{n-1}$. The function
  $$
    H_j(\xi'):= \lambda^{\frac{n-1}{2}} h_j(\psi_j(\xi'))(1-|\xi'|^2)^{-1/2}
  $$ 
  then satisfies $\supp(H_j)\subset B$ so that integration by parts  yields for all $|x|\geq 1$ and
  $j=1,\ldots,2n-2$
  \begin{align*}
     |I_j| 
     &= \left|\int_{B} H_j(\xi') e^{-i\sqrt\lambda|x|\xi_{n-1}}\,d\xi' \right|  \\
     &=  \left| \int_{\R^{n-1}} H_j(\xi') 
      \frac{\partial^m}{\partial (\xi_{n-1})^m} \Big( e^{- i\sqrt\lambda|x|\xi_{n-1}} \Big) \,d\xi'\right|
      \cdot |\sqrt\lambda x|^{-m}   \\
     &= \left| \int_{\R^{n-1}} \left(\frac{\partial^m}{\partial (\xi_{n-1})^m} H_j(\xi') \right) e^{-
     i\sqrt\lambda|x|\xi_{n-1}}\,d\xi'\right| \cdot |\sqrt\lambda x|^{-m}   \\
     &=   \int_{B} |\nabla^m H_j(\xi')| \,d\xi'  \cdot |\sqrt\lambda x|^{-m} \\
     &\leq C \|h\|_{C^m}  \cdot |x|^{-m}     
  \end{align*}
  Since the estimate for $|x|\leq 1$ is trivial and $m=\lfloor \frac{n-1}{2}\rfloor + 1 \geq \frac{n}{2}
  $, we conclude
  \begin{equation}\label{eq:stationary_phase_I}
    |I_j| \leq C \|h\|_{C^m}    (1+|x|)^{\frac{1-n}{2}-\alpha}  \quad\text{for all
    }x\in\R^n,\;j=1,\ldots,2n-2,\; \alpha\in (0,\frac{1}{4}).
 \end{equation}
   
  \medskip
  
  Next we analyze the integrals $I_{2n-1},I_{2n-2}$. With $\psi_j,H_j$ as above we define
  $$
    H_j^*(\eta):= H_j(\eta\sqrt{2-|\eta|^2})\cdot (2-2|\eta|^2)(2-|\eta|^2)^{\frac{n-3}{2}}
    \qquad (j=2n-1,2n).
  $$ 
  Again the supports of $H_j,H_j^*$ are contained in the interior of the unit ball $B$
  so that neither function is singular. Performing twice a change of coordinates we  get
 \begin{align*}
   I_{2n-1}  
   &= \int_{\R^{n-1}} H_{2n-1}(\xi') e^{-i\sqrt\lambda |x|\sqrt{1-|\xi'|^2}}  \,d\xi' 
   =  e^{-i\sqrt\lambda |x|} \int_{\R^{n-1}} H^*_{2n-1}(\eta) e^{i\sqrt\lambda |x||\eta|^2}  \,d\eta,  \\
   I_{2n}  
   &= \int_{\R^{n-1}} H_{2n}(\xi') e^{+i\sqrt\lambda |x|\sqrt{1-|\xi'|^2}}  \,d\xi' 
   =  e^{+i\sqrt\lambda |x|} \int_{\R^{n-1}} H^*_{2n}(\eta) e^{-i\sqrt\lambda |x||\eta|^2}  \,d\eta.  
 \end{align*}
 The integrals may be estimated using Proposition~11 in~\cite{Man_LAPperiodic} for
 $s:=m-2\alpha$. Notice that $\alpha\in (0,\frac{1}{4})$ ensures $s\geq \frac{n+1}{2}-2\alpha
 >\frac{n-1}{2}$ so that the estimate from this proposition is valid. 
 Using
 \begin{align*}
   H_{2n-1}^*(0) 
   &= 2^{\frac{n-1}{2}}H_{2n-1}(0) 
   = (2\lambda)^{\frac{n-1}{2}} h_{2n-1}(+\sqrt\lambda e_n)
   = \frac{1}{\sqrt{2\pi}} \left(\frac{\lambda}{\pi}\right)^{\frac{n-1}{2}} h(+\sqrt\lambda \hat x), \\
   H_{2n}^*(0) 
   &= 2^{\frac{n-1}{2}}H_{2n}(0) 
   = (2\lambda)^{\frac{n-1}{2}} h_{2n}(-\sqrt\lambda e_n)
   = \frac{1}{\sqrt{2\pi}} \left(\frac{\lambda}{\pi}\right)^{\frac{n-1}{2}} h(-\sqrt\lambda \hat x) 
 \end{align*}
 as well as $\|H_j\|_{H^m(\R^{n-1})}\leq C\|h\|_{C^m}$ for $j\in\{2n-1,2n\}$ we deduce from the first
 inequality in Proposition~11~\cite{Man_LAPperiodic}
 \begin{align}\label{eq:stationary_phase_II}
 \begin{aligned}
   \left|  
   I_{2n-1}  
   - e^{i(\frac{n-1}{4}\pi-\sqrt\lambda |x|)}
   \frac{1}{\sqrt{2\pi}} \left(\frac{\sqrt\lambda}{|x|}\right)^{\frac{n-1}{2}} h(\sqrt\lambda \hat x)
   \right| 
   &\leq  C \|h\|_{C^m} |x|^{\frac{1-n}{2}-\alpha}, \\ 
   \left| 
   I_{2n}  
   - e^{-i(\frac{n-1}{4}\pi-\sqrt\lambda |x|)} \frac{1}{\sqrt{2\pi}} \left(\frac{\sqrt\lambda}{|x|}\right)^{\frac{n-1}{2}} 
   h(-\sqrt\lambda \hat x)
   \right| 
   &\leq  C \|h\|_{C^m} |x|^{\frac{1-n}{2}-\alpha}.   
 \end{aligned}
 \end{align}
 Combining \eqref{eq:stationary_phase_I},\eqref{eq:stationary_phase_II} and 
 $\widehat{h\,d\sigma_\lambda}=I_1+\ldots+I_{2n}$  we find for $|x|\geq 1$
 \begin{align*}
   \left| \widehat{h\,d\sigma_\lambda}(x) 
   - \frac{1}{\sqrt{2\pi}} \left(\frac{\sqrt\lambda}{|x|}\right)^{\frac{n-1}{2}}
   m_h(x)
   \right|  
   \leq  C  \|h\|_{C^m} |x|^{\frac{1-n}{2}-\alpha},      
 \end{align*}
 where $m_h$ was introduced in~\eqref{eq:def_mh}. In view of the estimate
 $|\widehat{h\,d\sigma_\lambda}(x)|\leq C\|h\|_\infty$ for $|x|\leq 1$ we 
 derive the weaker statements 
 $$
   |\widehat{h\,d\sigma_\lambda}(x)|\leq C\|h\|_{C^m}(1+|x|)^{\frac{1-n}{2}}
 $$ 
 as well as  
 \begin{align*}
   &\frac{1}{R}\int_{B_R}  \left| \widehat{h\,d\sigma_\lambda}(x) 
   - \frac{1}{\sqrt{2\pi}} \left(\frac{\sqrt\lambda}{|x|}\right)^{\frac{n-1}{2}}
   m_h(x)
   \right|^2 \,dx \\
   &\leq C\|h\|_\infty^2 \cdot 
   \frac{1}{R}\int_{B_1} (1+|x|^{1-n})\,dx
   + C \|h\|_{C^m}^2 \cdot \frac{1}{R} \int_{B_R\sm B_1} |x|^{1-n-2\alpha}\,dx  \\
   &\leq C\Big(\frac{1}{R} + \frac{1}{R^{2\alpha}}\Big) \|h\|_{C^m}^2  \quad 
   \to 0 \qquad \text{as }R\to\infty. 
 \end{align*} 
 This finishes the proof of Proposition~\ref{prop:Herglotz_waves}. \qed
  
\section{Appendix B: Proof of Theorem~\ref{thm:LAP_NLM}}
  
  In this section we prove the Limiting Absorption Principle from Theorem~\ref{thm:LAP_NLM}. We recall the
  statement for the convenience of the reader.
  
\begin{thm*}  
  Let $\lambda>0$ and assume that $t,q\in (1,\infty)$ satisfy~\eqref{eq:admissible}. Then there is a bounded
  linear operator $R_\lambda:L^t(\R^3;\R^3)\cap L^q(\R^3;\R^3)\to L^q(\R^3;\R^3)$ such that $R_\lambda G\in
  H_{loc}(\curl;\R^3)$ is a weak solution of $\nabla\times\nabla\times E-\lambda E = G$ 
  provided $G\in L^t(\R^3;\R^3)\cap L^q(\R^3;\R^3)$. Moreover, we have
  \begin{equation*}
    \|R_\lambda G\|_q \leq C(\|G_1\|_q+\|G_2\|_t) \leq C(\|G\|_q+\|G\|_t) 
  \end{equation*} 
  and $R_\lambda G = -\frac{1}{\lambda}G_1 + \fR_\lambda G_2$ for $\fR_\lambda$ from Theorem~\ref{thm:LAP}
  (applied componentwise).
  If $G\in L^t(\R^3;\R^3)$ is cylindrically symmetric then so is $R_\lambda G$ and $R_\lambda G\in
   W_{loc}^{2,q}(\R^3;\R^3)$ is a strong solution satisfying $\|R_\lambda G\|_q \leq C\|G\|_t$. 
\end{thm*}

 To prove this, we first show that the domain of the selfadjoint realization of the curl-curl operator $LE:=
 \nabla\times\nabla\times E$ is given by 
 \begin{align*}
   \mathcal D
   &:= \{E\in L^2(\R^3;\R^3): L E\in L^2(\R^3;\R^3)\} \\
   &:= \{E\in L^2(\R^3;\R^3): \xi\mapsto |\xi|^2\hat E(\xi)-\skp{\hat E(\xi)}{\xi}\xi\in L^2(\R^3;\R^3)\}
 \end{align*}
 and that its spectrum is $[0,\infty)$. Using the Helmholtz decomposition 
 \begin{equation}\label{eq:Helmholtz_decomp}
   \hat G_1(\xi) := \skp{\hat G(\xi)}{\frac{\xi}{|\xi|}}\frac{\xi}{|\xi|},\qquad  
   \hat G_2(\xi) := \hat G(\xi) - \hat G_1(\xi)
 \end{equation}
 of a vector field $G\in L^2(\R^3;\R^3)$ we get the following.
 
 \begin{prop}\label{prop:selfadjointness}
   The curl-curl operator $L: \mathcal D\to L^2(\R^3;\R^3)$ is selfadjoint with
   spectrum $\sigma(L)=[0,\infty)$. For $\mu\in\C\sm {[0,\infty)}$ the resolvent
   is given by 
   $$
     (L - \mu)^{-1}G = -\frac{1}{\mu}G_1 + \mathcal R(\mu)G_2
   $$
   where $\mathcal R(\mu)$ is (in each component) the operator defined at
   the beginning of Section~2.
 \end{prop} 
 \begin{proof}
   The curl-curl-operator $L$ is symmetric when defined on the Schwartz functions in $\R^3$. So we have to
   show that all $E\in L^2(\R^3;\R^3)$ in the domain of its adjoint actually belong to $\mathcal  D$.
   So assume that $E\in L^2(\R^3;\R^3)$ satisfies for all
   $F\in \mathcal D$ the inequality 
   $$
     \left| \int_{\R^3} \skp{E}{LF} \right| \leq C \|F\|_2.
   $$
   Using Plancherel's identity on both sides  this can be rewritten as
   $$
     \left| \int_{\R^3} \skp{|\xi|^2\hat E(\xi)-\skp{\hat E(\xi)}{\xi}\xi}{\ov{\hat F(\xi)}} \,d\xi\right|
     = \left| \int_{\R^3} \skp{\hat E(\xi)}{|\xi|^2\ov{\hat F(\xi)}-\skp{\ov{\hat F(\xi)}}{\xi}\xi} \,d\xi
     \right| \leq C  \|\hat F\|_2.
   $$
   Since this holds for all $F\in\mathcal D$, which is dense in $L^2(\R^3;\R^3)$, we infer that $\xi\mapsto
   |\xi|^2\hat E(\xi)-\skp{\hat E(\xi)}{\xi}\xi$ is square-integrable, which precisely means $E\in \mathcal D$. 
   
   \medskip
   
   To prove the second claim let $\mu\in\C\sm {[0,\infty)}$ and $G\in L^2(\R^3;\R^3)$. Then $\nabla
   \times\nabla\times E - \mu E = G$ is equivalent to 
   $$
     |\xi|^2\hat E(\xi) -\skp{\xi}{\hat E(\xi)}\xi - \mu \hat E(\xi) = \hat G(\xi).
   $$
   Decomposing $E$ into $E_1,E_2$ as in~\eqref{eq:Helmholtz_decomp} and multiplying the above equation with
   $\xi/|\xi|\in\R^3$ we find
  \begin{equation*} 
    \hat E_1(\xi) 
    =  -\frac{1}{\mu} \skp{\hat G(\xi)}{\frac{\xi}{|\xi|}}\frac{\xi}{|\xi|} 
    =  -\frac{1}{\mu} \hat G_1(\xi).
  \end{equation*} 
  This implies
  \begin{align*}
    \hat E_2(\xi) (|\xi|^2-\mu) 
    = \hat G(\xi) + \mu\hat E_1(\xi) 
    = \hat G(\xi) - \skp{\hat G(\xi)}{\frac{\xi}{|\xi|}} \frac{\xi}{|\xi|}
    = \hat G_2(\xi). 
  \end{align*}
  So we have
  \begin{align*}
    E = E_1+E_2  
    = -\frac{1}{\mu}  G_1 + \mathcal F^{-1} \left(\frac{\hat G_2(\cdot)}{|\cdot|^2-\mu} \right)
    = -\frac{1}{\mu}  G_1 + \mathcal R(\mu)G_2.
  \end{align*}
  Since the right hand side defines a bounded linear operator provided $\mu\in\C\sm {[0,\infty)}$, this proves
  that $\C\sm {[0,\infty)}$ belongs to the resolvent set of the curl-curl operator. 
  
  \medskip
  
  By the closedness of the
  spectrum it therefore suffices to show that for all $\mu\in (0,\infty)$ there is a Weyl sequence of the
  curl-curl operator. Indeed, as in the case of Laplacian one may consider the sequence 
  $$
    F_n(x):= c_n\chi(x/n)F(x)\qquad\text{where }
    F(x):=  \vecIII{\cos(\sqrt\mu  x_3)}{0}{0}
  $$
  where $\chi\in C_0^\infty(\R^3)$ is identically one on the cuboid $W:=[-1,1]^3\subset\R^3$ and zero outside
  $2W$. The factor  $c_n>0$ is a normalizing constant ensuring $\|F_n\|_2=1$. Using 
  \begin{align*}
    \int_{-a}^a \int_{-a}^a\int_{-a}^a \cos^2(n\sqrt\mu x_3) \,dx_3\,dx_2\,dx_1 
    &=4a^3 + o(1)\qquad (n\to\infty), \\
    \int_{-a}^a \int_{-a}^a\int_{-a}^a \sin^2(n\sqrt\mu x_3) \,dx_3\,dx_2\,dx_1 
    &=4a^3 + o(1)\qquad (n\to\infty)  
  \end{align*}
  for $a=1$ we first get 
   $$
     c_n 
     = \frac{1}{\|\chi(\cdot/n)F\|_2}  
     = \frac{1}{n^{3/2} \|\chi F(n\cdot)\|_2}
    \leq \frac{1}{n^{3/2}\|F(n\cdot)\|_{L^2(W)}}
     \leq  \frac{1}{n^{3/2} (2+o(1))}.
   $$
  Using this as well as the above asymptotics for $a=2$ we find
  \begin{align*}
    \|LF_n-\mu F_n\|_2
    &= \|-\Delta  F_n-\mu F_n\|_2 \\
    &= c_n\cdot \Big\|   \chi(x/n) (-\Delta F-\mu F) - \frac{2}{n}
    \nabla\chi(x/n) \cdot \nabla F(x) 
    - \frac{1}{n^2}\Delta\chi(x/n) F(x) \Big\|_2 \\
    &\leq  Cn^{-3/2} \cdot \Big( 0 + 2n^{1/2} \|\nabla\chi\cdot\nabla F(n\cdot)\|_2 
    + n^{-1/2} \|(\Delta\chi) F(n\cdot)\|_2\Big) \\
    &\leq C \Big( n^{-1} \| |\nabla F(n\cdot)|\|_{L^2(2W)} + n^{-2}\| |F(n\cdot)|\|_{L^2(2W)} \Big) \\
    &\leq  C \big( \sqrt{\mu} n^{-1} +n^{-2}\big) (\sqrt{4\cdot 2^3}+o(1)) \\
    &=O(n^{-1}) \qquad\text{as }n\to\infty 
  \end{align*}
  so that a Weyl sequence at the level $\mu$ is found.
 \end{proof}
  
 \medskip
 
 \noindent
 \textbf{Proof of Theorem~\ref{thm:LAP_NLM}:}  In order to construct $R_\lambda$ consider a Schwartz function
 $G\in \mathcal S(\R^3;\R^3)$ so that $E^\eps:=(L - \lambda-i\eps)^{-1}G$ is
 well-defined and Proposition~\ref{prop:selfadjointness} implies 
 $$
   E^\eps = -\frac{1}{\lambda+i\eps}  G_1 + \mathcal R(\lambda+i\eps)G_2 \in \mathcal D+i\mathcal D. 
 $$ 
 By the Limiting Absorption Principle for the Helmholtz operator
 from Theorem~\ref{thm:LAP} we get
  \begin{align*} 
    (L - \lambda-i0)^{-1}G
    := \lim_{\eps\to 0^+} E^\eps 
     =   - \frac{1}{\lambda}  G_1  + \mathcal R(\lambda+i0)G_2.
  \end{align*}
  Here, both limits exist in $L^q(\R^3;\R^3)$. Indeed, the Riesz transform maps $L^r(\R^3)$ into itself for
  all $r\in (1,\infty)$ and thus $\|G_1\|_r\leq C\|G\|_r$ for $r\in\{t,q\}$,  hence $\|G_2\|_t\leq
  C\|G\|_t<\infty$ and $\|G_1\|_q\leq C\|G\|_q<\infty$.  
  Taking the real part of this we obtain that 
  $$
    R_\lambda G:= - \frac{1}{\lambda}  G_1  + \fR_\lambda G_2
  $$ 
  defines a  bounded linear operator from $L^t(\R^3;\R^3)\cap L^q(\R^3;\R^3)$ to $L^q(\R^3;\R^3)$  with
  \begin{align*}
    \|R_\lambda G\|_q
    \leq \frac{1}{\lambda} \|G_1\|_q
    + \|\fR_\lambda G_2\|_q 
    \leq C(\|G_1\|_q + \|G_2\|_t).    
  \end{align*}

  \medskip
  
  Next we prove that $R_\lambda G$ is a weak solution lying in $H_{loc}(\curl;\R^3)$. We will use
  that $E^\eps\in L^q(\R^3;\R^3)$ implies $E^\eps\in L^{q'}_{loc}(\R^3;\R^3)$ due to $q>2>q'$.   
  Testing the equation for $E^\eps\in \mathcal D+i\mathcal D$ with $\ov{E^\eps}\phi^2 \in H(\curl;\C^3)$ we
  get
  \begin{equation}\label{eq:equation_for_Eeps}
    \int_{\R^3} \skp{\nabla\times E^\eps}{\nabla \times (\ov{E^\eps}\phi^2)} - \lambda  
    |E^\eps|^2\phi^2 = \int_{\R^3} \skp{G}{\ov{E^\eps}}\phi^2
    \qquad\text{for all }\phi\in C_0^\infty(\R^3). 
  \end{equation} 
  So for any given compact set $K\subset\R^3$ we may choose a nonnegative test function $\phi$   such  
  that $\phi|_K\equiv 1$, set $K':=\supp(\phi)$. Then we get 
  $$
    \int_{\R^3} \skp{G}{\ov{E^\eps}}\phi^2
    \leq \|G\|_{L^{q'}(K';\R^3)} \|E^\eps\|_{L^q(K';\C^3)} \|\phi\|_\infty^2
  $$
  as well as  
  \begin{align*}
    \int_{\R^3} \skp{\nabla\times E^\eps}{\nabla \times (\ov{E^\eps}\phi^2)} - \lambda  
    |E^\eps|^2\phi &= \int_{\R^3} |\nabla\times (E^\eps\phi)|^2  - |\nabla\phi\times E^\eps|^2
    - \lambda  |E^\eps|^2\phi^2 \\
    &\geq \int_{\R^3} |\nabla\times (E^\eps\phi)|^2  -  \big(|\nabla\phi|^2+ \lambda \phi^2 \big)|E^\eps|^2 \\
     &\geq \int_{\R^3} |\nabla\times (E^\eps\phi)|^2  - 
     \||\nabla\phi|^2+\lambda\phi^2\|_{L^{\frac{q}{q-2}}(K';\R^3)} \|E^\eps\|_{L^q(K';\R^3)}^2.
  \end{align*}
  Here we used $|a\times b|\leq |a||b|$ for $a,b\in\C^3$ and $q>\frac{2n}{n-1}>2$.  
  Combining the previous two inequalities with~\eqref{eq:equation_for_Eeps} we get
  \begin{align*}
    \int_{K} |\nabla\times E^\eps|^2
    &\leq \int_{\R^3} |\nabla\times (E^\eps\phi)|^2  \\
    &\leq  C \left(\|E^\eps\|_{L^q(K';\C^3)}^2  + \|G\|_{L^{q'}(K';\R^3)} \|E^\eps\|_{L^q(K';\C^3)}\right) \\
    &\leq  C \left(\|E^\eps\|_{L^q(\R^3;\C^3)}^2  + \|G\|_{L^{q'}(K';\R^3)} \|E^\eps\|_{L^q(\R^3;\C^3)}\right)
    \\
    &\leq C<\infty
  \end{align*}
  because the functions $E^\eps$ are equibounded in $L^q(\R^3;\C^3)$. The latter fact comes from the
  proof of Guti\'{e}rrez' Limiting Absorption Principle, see Theorem~6 in~\cite{Gut_nontrivial}. 
  Since $K$ was arbitrary, we conclude that $(E^\eps)$ is bounded in $H_{loc}(\curl;\C^3)$ and therefore the
  weak limit of its real part $R_\lambda G$ also belongs to that space. 
  
  \medskip
  
  Finally we show that $R_\lambda$ leaves the space of cylindrically symmetric solutions invariant. If $F\in
  \mathcal S(\R^3;\R^3)$ is cylindrically symmetric, then there is $F_0:[0,\infty)\times\R\to \R$ such that 
  $$
    F(x_1,x_2,x_3)= F_0(\sqrt{x_1^2+x_2^2},x_3) \vecIII{-x_2}{x_1}{0}
    \qquad\text{for all } (x_1,x_2,x_3)\in\R^3.
  $$
  It suffices to show that there is $\tilde F_0:[0,\infty)\times\R\to\C$ such that
  $\ov{\tilde F_0(r,-z)}=-\tilde F_0(r,z)$ for all $r\geq 0,z\in\R$ and 
  \begin{equation} \label{eq:form_hatF}
   \hat F(\xi_1,\xi_2,\xi_3)
   = \tilde F_0(\sqrt{\xi_1^2+\xi_2^2},\xi_3) \vecIII{-\xi_2}{\xi_1}{0}
    \qquad\text{for all } (\xi_1,\xi_2,\xi_3)\in\R^3.
  \end{equation}
  Clearly the third component of $\hat F$ vanishes identically. Using the symmetry
  of $F$ and the definition of the Fourier transform we moreover obtain after some calculations that for all
  $\theta\in [0,2\pi)$ we have 
  $$
    \bigskp{\hat F(\xi_1,\xi_2,\xi_3)}{
    \matIII{\cos(\theta)}{-\sin(\theta)}{0}{\sin(\theta)}{\cos(\theta)}{0}{0}{0}{0} 
     \vecIII{\xi_1}{\xi_2}{0}}
    = \rho_\theta(\sqrt{\xi_1^2+\xi_2^2},\xi_3) 
  $$ 
  for some function $\rho_\theta:[0,\infty)\times\R\to\C$, i.e., for every given $\theta\in [0,2\pi)$ the left
  hand side is invariant under all rotations with respect to the $(\xi_1,\xi_2)$-variable. Using this for
  $\theta=0$ we get
  \begin{align*}
    &\bigskp{\hat F(\xi_1,\xi_2,\xi_3)}{ \vecIII{\xi_1}{\xi_2}{0}} \\
    &= \rho_0(\sqrt{\xi_1^2+\xi_2^2},\xi_3) \quad
    =\quad  \bigskp{\hat F(0,\sqrt{\xi_1^2+\xi_2^2},\xi_3)}{ \vecIII{0}{\sqrt{\xi_1^2+\xi_2^2}}{0}} \\
    &= \frac{1}{(2\pi)^{3/2}} \int_{\R^3} F_0(\sqrt{x_1^2+x_2^2},x_3) \bigskp{\vecIII{-x_2}{x_1}{0}}{
    \vecIII{0}{\sqrt{\xi_1^2+\xi_2^2}}{0}} e^{-i(x_2\sqrt{\xi_1^2+\xi_2^2}+x_3\xi_3)} \,dx \\
    &= \frac{1}{(2\pi)^{3/2}} \int_{\R^2} \left( \int_\R F_0(\sqrt{x_1^2+x_2^2},x_3) x_1
     \,dx_1\right) \sqrt{\xi_1^2+\xi_2^2} e^{-i(x_2\sqrt{\xi_1^2+\xi_2^2}+ x_3\xi_3)} \,d(x_2,x_3) \\
    &= 0
  \end{align*}
  because $x_1\mapsto F_0(\sqrt{x_1^2+x_2^2},x_3) x_1$ is odd for all fixed
  $(x_2,x_3)\in\R^2$. This shows 
  $$
    \hat F(\xi_1,\xi_2,\xi_3) = \frac{\rho_{\pi/2}(\sqrt{\xi_1^2+\xi_2^2},\xi_3)}{\xi_1^2+\xi_2^2}
    \vecIII{-\xi_2}{\xi_1}{0} \qquad\text{for all }(\xi_1,\xi_2,\xi_3)\in\R^3 
  $$ 
  so that $\hat F$ can be written in the form~\eqref{eq:form_hatF}. Finally, since $F$ is real-valued
  we get $\ov{\hat F(-\xi)}=\hat F(\xi)$ and hence $\ov{\tilde F_0(r,-z)}=-\tilde F_0(r,z)$ for all
  $r\geq 0,z\in\R$. This finishes the proof. \qed

\section{Appendix C: The Ruiz-Vega resolvent estimates}
 
 In this section we review the resolvent estimates by Ruiz and Vega that are essentially contained in
 Theorem~3.1 in~\cite{RuVe_OnLocal}. Since this theorem does not exactly provide the estimates we need, we decided to
 reformulate their results in the way we apply them in the proof of Proposition~\ref{prop:LAP_asymptotics}. We
 even show a bit more.
 
 \begin{thm}\label{thm:RuizVega}
   Let $n\in\N,n\geq 2$ and assume $f\in L^q(\R^n)$ for $\frac{1}{n+1}\leq \frac{1}{q}-\frac{1}{2}\leq
   \min\{\frac{1}{2},\frac{2}{n}\}$ with $(n,q)\neq (4,1)$.
   Then there is a $C>0$ such that for all $\eps\neq 0$ we have 
    \begin{equation*} 
      \sup_{R\geq 1} \left(\frac{1}{R}\int_{B_R} |\mathcal R(\lambda+i\eps)f(x)|^2\,dx\right)^{1/2} \leq C 
      \|f\|_q.
    \end{equation*}
   If moreover $\frac{1}{n+1}\leq \frac{1}{q}-\frac{1}{2}\leq \frac{1}{n},(q,n)\neq (1,2)$ holds, then
   \begin{equation*} 
      \sup_{R\geq 1} \left(\frac{1}{R}\int_{B_R} |\nabla \mathcal R(\lambda+i\eps)f(x)|^2\,dx\right)^{1/2}
      \leq C \|f\|_q.
   \end{equation*}
 \end{thm}
 
 We emphasize that this result implies that the inequality~\eqref{eq:estimate_asymptotics} from
 the proof of Proposition~\ref{prop:LAP_asymptotics} holds for $q:=p'$ because $\frac{2(n+1)}{n-1}\leq p\leq
 \frac{2n}{(n-4)_+}$, $(n,p)\neq (4,\infty)$ is equivalent to 
 $\frac{1}{n+1}\leq \frac{1}{q}-\frac{1}{2}\leq \min\{\frac{1}{2},\frac{2}{n}\}$, $(n,q)\neq (4,1)$.  
 It seems that our statement dealing with the two-dimensional case $n=2$ does not appear in the literature. 
 In the case $n\geq 3$ our Theorem~\ref{thm:RuizVega} is covered by 
 Guti\'{e}rrez' Theorem 7 in~\cite{Gut_nontrivial} except for the endpoint case $n=3,q=1$. The
 proof by Guti\'{e}rrez is not carried out in detail but it is referred to the paper of Ruiz and
 Vega (Theorem~3.1 in~\cite{RuVe_OnLocal}) where a closely related but different result is proved. 
 So we believe that an updated and self-contained version of these resolvent estimates might be useful even
 though our proof below mainly reformulates the arguments of Ruiz and Vega.

 \medskip\medskip
 
 \noindent \textbf{Proof of Theorem~\ref{thm:RuizVega}:}\;  It suffices to prove the estimates for
 Schwartz functions $f\in \mathcal S(\R^n)$ and, via rescaling, for $\lambda=1$. Then we have  
 $$
  \mathcal R(1+i\eps)f
  = \mathcal F^{-1}\left(  \frac{1}{|\xi|^2-1-i\eps} \hat f(\xi) \right). 
 $$
We introduce the splitting $\mathcal R(1+i\eps)f=v^\eps+w^\eps$ where 
$$
  v^\eps:=\mathcal F^{-1}\left(  \frac{\phi(\xi)}{|\xi|^2-1-i\eps} \hat f(\xi) \right),\qquad 
  w^\eps:=\mathcal F^{-1}\left(  \frac{1-\phi(\xi)}{|\xi|^2-1-i\eps} \hat f(\xi)\right)
$$ 
and $\phi\in C_0^\infty(\R^n)$ is a test function satisfying $\supp(\phi)\subset B_{1+\delta}\sm
B_{1-\delta}$, $\delta:=1-\frac{1}{\sqrt{1+\frac{1}{2n}}}$ as well as $\phi\equiv 1$ on a neighbourhood of
the unit sphere. Then we have $w^\eps=G^\eps \ast f$ where 
$$
  |G^\eps(z)|\leq C\begin{cases}
     \min \{|z|^{2-n},|z|^{-1-n}\} &,n\geq 3 \\
     \min \{|\log(z)|,|z|^{-3}\} &,n=2  
  \end{cases}\qquad\text{and}\qquad 
  |\nabla G^\eps(z)|\leq C\min\{|z|^{1-n},|z|^{-1-n}\} 
$$ 
for some $C>0$ independent of $\eps$, see page 8-9 in~\cite{BoCaMa_4thorder} for related estimates. 
These bounds imply $G^\eps\in L^{\frac{2q}{3q-2}}(\R^n)$ whenever $0\leq \frac{1}{q}-\frac{1}{2}\leq \min\{\frac{1}{2},\frac{2}{n}\}$ with
$\frac{1}{q}-\frac{1}{2}<\frac{2}{n}$ and the corresponding norms are uniformly bounded from above with
respect to $\eps$. For such $q$ we get
\begin{align*}
  \sup_{R\geq 1}\left( \frac{1}{R}\int_{B_R} |w^\eps (x)|^2 \,dx \right)^{1/2}
  \leq  \|w^\eps\|_2
  = \|G^\eps\ast f\|_2
  \leq \|G^\eps\|_{\frac{2q}{3q-2}}\|f\|_q
  \leq C\|f\|_q.
\end{align*}
It remains to prove this estimate for $n\geq 5$ and $\frac{1}{q}-\frac{1}{2}= \frac{2}{n}$. In this case 
we use $G^\eps\in L^{\frac{2q}{3q-2},w}(\R^n)$ with uniformly bounded norms so that Young's convolution
inequality for classical and weak Lebesgue spaces (see Theorem 1.4.25 in~\cite{Grafakos} for the latter)
yields
\begin{align*}
  \sup_{R\geq 1}\left( \frac{1}{R}\int_{B_R} |w^\eps (x)|^2 \,dx \right)^{1/2}
  \leq  \|G^\eps\ast f\|_2
  \leq \|G^\eps\|_{\frac{2q}{3q-2},w}\|f\|_q
  \leq C\|f\|_q.
\end{align*}
(Notice that we also have $G^\eps\in L^{\frac{2q}{3q-2},w}(\R^n)$ 
 in the case $\frac{1}{q}-\frac{1}{2}= \frac{2}{n}$ and $n=4$, but Theorem~1.4.25~\cite{Grafakos} does not
 apply since each of the exponents $2,\frac{2q}{3q-2},q$ has to be different from 1 or $\infty$.)
The same way, if  $0\leq \frac{1}{q}-\frac{1}{2}<\frac{1}{n}$ then we have
$|\nabla G^\eps|\in L^{\frac{2q}{3q-2}}(\R^n)$ and $\frac{1}{q}-\frac{1}{2}=\frac{1}{n}$ implies
$|\nabla G^\eps|\in L^{\frac{2q}{3q-2},w}(\R^n)$  with uniformly bounded norms. In the former case we get
\begin{align*}
  \sup_{R\geq 1}\left( \frac{1}{R}\int_{B_R} |\nabla w^\eps (x)|^2 \,dx \right)^{1/2}
\leq  \||\nabla w^\eps|\|_2
= \||\nabla G^\eps|\ast f\|_2
\leq \||\nabla G^\eps|\|_{\frac{2q}{3q-2}}\|f\|_q
\leq C\|f\|_q.
\end{align*} 
and in the latter case we have under the additional assumption $(q,n)\neq (1,2)$ (for the same reason as
above)
\begin{align*}
  \sup_{R\geq 1}\left( \frac{1}{R}\int_{B_R} |\nabla w^\eps (x)|^2 \,dx \right)^{1/2}
\leq   \||\nabla G^\eps|\ast f\|_2
\leq \||\nabla G^\eps|\|_{\frac{2q}{3q-2},w} \|f\|_q
\leq C\|f\|_q.
\end{align*} 
So it remains to  show that the estimate   
\begin{equation}\label{eq:estimate_veps} 
  \sup_{R\geq 1}\left(\frac{1}{R}\int_{B_R} |v^\eps (x)|^2+|\nabla v^\eps(x)|^2 \,dx \right)^{1/2} \leq C
  \|f\|_q
\end{equation}
holds whenever $\frac{1}{q}-\frac{1}{2}\geq \frac{1}{n+1}$. 

\medskip

 To prove \eqref{eq:estimate_veps}  we split $v^\eps$ into $2n$ different pieces using a suitable partition of
 unity. In view of $\supp(\hat{v^\eps})\subset B_{1+\delta}\sm B_{1-\delta}$ we consider the covering
 $\{O_{1,+},O_{1,-},\ldots,O_{n,+},O_{n,-}\}$ of $B_{1+\delta}\sm B_{1-\delta}$ given by the open sets 
  $$
    O_{j,\pm}:= \left\{\xi\in B_{1+\delta}\sm B_{1-\delta}:  \pm \xi_j> \frac{1}{\sqrt{2n}}|\xi| \right\}
  $$
  Let $\{\eta_{1,+},\eta_{1,-},\ldots,\eta_{n,+},\eta_{n,-}\}$ be an associated partition of unity so that 
  $$
    v = \sum_{j=1}^n (v^\eps_{j,+}+v^\eps_{j,-}) \quad
    \text{where}\quad 
    v^\eps_{j,\pm} := \mathcal F^{-1}\left( \frac{\eta_{j,\pm}(\xi)\phi(\xi)}{|\xi|^2-1-i\eps} \hat
    f(\xi)\right)  
  $$
  The reason for this is that we want to make use of the following inequalities: 
  \begin{align}\label{eq:partition_of_unity}
    \xi\in\supp(\hat v^\eps_{j,\pm})
    \quad\Longrightarrow\quad
    \begin{cases}
    \quad |\xi_j|\leq |\xi|\leq  1+\delta,\\
    \quad
    \pm \xi_j>\frac{1}{\sqrt{2n}}|\xi|>\max\Big\{\frac{1-\delta}{\sqrt{2n}},\frac{1}{\sqrt{2n}}|\xi_j'|\Big\},
    \\
    \quad  |\xi_j'|  
    \leq \frac{\sqrt{|\xi_j'|^2+|\xi_j|^2}}{\sqrt{1+\frac{1}{2n}}}
    \leq \frac{1+\delta}{\sqrt{1+\frac{1}{2n}}}
    = 1-\delta^2.
    \end{cases}
    \begin{aligned}
  \end{aligned}
  \end{align}
  Here we used the notation $\xi=(\xi_1,\ldots,\xi_n)\in\R^n$ and
  $\xi_j':=(\xi_1,\ldots,\xi_{j-1},\xi_{j+1},\ldots,\xi_n)$. Since the estimates for $v_{j,\pm}$ are the same
  for all $j\in\{1,\ldots,n\}$ up to textual modifications we only consider $j=1$. Moreover, the estimates for
  $v_{1,+},v_{1,-}$ only differ at one point (which we will mention), so that   we only carry out the
  estimates for $v^\eps_{1,+}$. To simplify the notation we write
  $v^\eps,\eta,\xi'$ instead of $v^\eps_{1,+},\eta_{1,+},\xi_1'=(\xi_2,\ldots,\xi_n)$.
 
 \medskip
  
 Using $B_R\subset [-R,R]\times\R^{n-1}$   we get
\begin{align} \label{eq:est1}
  \begin{aligned}
  \frac{1}{R}\int_{B_R} |v^\eps (x)|^2+ |\nabla v^\eps (x)|^2 \,dx
  &\leq \frac{1}{R}\int_0^R  \Big( \int_{\R^{n-1}}|v^\eps(x_1,x')|^2+|\nabla v^\eps(x_1,x')|^2\,dx'  \Big)
  \,dx_1  \\
  &= \frac{1}{R}\int_0^R \|(A^\eps\ast f)(x_1,\cdot)\|_{L^2(\R^{n-1})}^2 \,dx_1
  \end{aligned}
\end{align}
 where $A^\eps:=  \mathcal F^{-1}\left( \frac{\eta(\xi)\phi(\xi)(1+|\xi|^2)^{1/2}}{|\xi|^2-1-i\eps}\right)$.
 We first provide some estimates related to this function. We have  
 \begin{align}\label{eq:defPsiteps}
   \begin{aligned}
   \Psi_t^\eps(\xi')
   &:= e^{-i\sqrt{1-|\xi'|^2}t} \mathcal{F}_{\xi_1}^{-1} (\hat A^\eps(\xi_1,\xi'))(t) \\
   &= \frac{1}{\sqrt{2\pi}} \int_\R 
   \frac{\eta(\xi_1,\xi')\phi(\xi_1,\xi')\sqrt{1+|\xi_1|^2+|\xi'|^2}}{\xi_1^2+|\xi'|^2-1-i\eps}
   e^{i(\xi_1-\sqrt{1-|\xi'|^2})t} \,d\xi_1 \\
   &=   \int_\R \psi^\eps(\xi_1,\xi') \cdot \frac{e^{i\xi_1t}}{\xi_1}\,d\xi_1 
    \end{aligned}
 \end{align}
 where
 \begin{align*}
   \psi^\eps(\xi_1,\xi') 
   := 
   \frac{\eta(\xi_1+\sqrt{1-|\xi'|^2},\xi')\phi(\xi_1+\sqrt{1-|\xi'|^2},\xi')\sqrt{2+2\xi_1\sqrt{1-|\xi'|^2}+|\xi_1|^2}}{\sqrt{2\pi}(\xi_1+
   2\sqrt{1-|\xi'|^2} - i\eps/\xi_1)}.
 \end{align*}
 In \eqref{eq:defPsiteps} we used the coordinate transformation $\xi_1\mapsto \xi_1+\sqrt{1-|\xi'|^2}$, which
 is well-defined because of $|\xi'|\leq 1-\delta^2$ by~\eqref{eq:partition_of_unity}. Moreover,
 $$
   \xi_1+2\sqrt{1-|\xi'|^2} 
   > \xi_1 
   \stackrel{\eqref{eq:partition_of_unity}}{\geq} \frac{1-\delta}{\sqrt{2n}}>0
   \qquad\text{for all }\xi\in \supp(\psi^\eps), \eps\in\R. 
 $$
 (In the estimates for $v_{1,-}$ the coordinate transformation to be used is  $\xi_1\mapsto
 \xi_1-\sqrt{1-|\xi'|^2}$ and the above estimate has to be replaced by
 $\xi_1-2\sqrt{1-|\xi'|^2}<\xi_1<-\frac{1-\delta}{\sqrt{2n}}<0$.) So we conclude that $\psi^\eps$ is smooth
 for every fixed $\eps\in\R$. Based on these properties of $\psi^\eps$ we now provide some estimates for $\Psi_t^\eps$.
 
 \medskip
 
 From~\eqref{eq:partition_of_unity} we infer $\eta(s,\xi')=0$ whenever $|\xi'|\geq 1-\delta^2,s\in\R$ so that
 \begin{equation}\label{eq:estimatesPsiteps1}
   \supp(\Psi_t^\eps)\subset B_{1-\delta^2}\qquad\text{for all }t\in\R,\eps\neq 0.
 \end{equation} 
 Moreover, we have for all $m\in\N_0$ and $\xi'\in\R^{n-1},|\xi'|\leq 1-\delta^2$
\begin{align}\label{eq:estimatesPsiteps2}
  \begin{aligned}
   \left| \nabla^m \Psi_t^\eps(\xi') \right| 
   &\stackrel{\eqref{eq:defPsiteps}}{=} \left| e^{i\sqrt{1-|\xi'|^2}t}  \mathcal
   F_1\left(\nabla^m_{\xi'}(\psi^\eps(\cdot,\xi'))\cdot \pv
   \left(\frac{1}{\cdot}\right) \right)(t) \right|\\
   &\leq \left\|\mathcal F_1( \nabla^m_{\xi'}(\psi^\eps(\cdot,\xi')))\ast \mathcal
   F_1\left(\pv\left(\frac{1}{\cdot}\right)\right) \right\|_{L^\infty(\R)} \\
   &\leq \|\mathcal F_1( \nabla^m_{\xi'}(\psi^\eps(\cdot,\xi')))\|_{L^1(\R)}
   \|i\pi\sign(\cdot)\|_{L^\infty(\R)}  \\
    &\leq \pi \|\mathcal F_1(\nabla^m_{\xi'}(\psi^\eps(\cdot,\xi')))\cdot
    (1+|\cdot|^2)(1+|\cdot|^2)^{-1}\|_{L^1(\R)}  \\
    &\leq \pi \|\mathcal F_1( \nabla^m_{\xi'}(-\partial_{\xi_1\xi_1}+1)(\psi^\eps(\cdot,\xi')))\|_{L^\infty(\R)} \|(1+|\cdot|^2)^{-1}\|_{L^1(\R)} \\
    &\leq C  \|  \nabla^m_{\xi'}(-\partial_{\xi_1\xi_1}+1)(\psi^\eps(\cdot,\xi'))\|_{L^1(\R)} \\
    &\leq C_m.
  \end{aligned}   
\end{align}  
From \eqref{eq:estimatesPsiteps1} and \eqref{eq:estimatesPsiteps2} we conclude that for all $m\in\N_0$ there
is a $C_m>0$ such that
\begin{equation}\label{eq:esimatePsiteps3}
  |\nabla^m \Psi_t^\eps(\xi')|\leq C_m(1+|\xi'|)^{-m} \qquad\text{for all }\xi'\in\R^{n-1},t\in\R,\eps\neq 0.
\end{equation}
 
 \medskip
 
 In view of~\eqref{eq:est1} we now use~\eqref{eq:esimatePsiteps3} in order to estimate the term
 \begin{align*} 
  \|(A^\eps\ast f)(x_1,\cdot) \|_{L^2(\R^{n-1})}^2
  &= \int_{\R^{n-1}} (A^\eps\ast f)(x_1,y') (\bar A^\eps\ast \bar f)(x_1,y') \,dy'  \\
  &= \int_{\R} \int_\R \Big( \int_{\R^{n-1}} f(z_1,y') S^\eps_{x_1,z_1,\tilde z_1}(f(\tilde z_1,\cdot))(y')
  \,dy' \Big) \,dz_1 \,d\tilde z_1
\end{align*}
for any given $x_1,z_1,\tilde z_1\in\R,\eps\neq 0$ where $S^\eps_{x_1,z_1,\tilde
z_1}:\mathcal{S}(\R^{n-1})\to\mathcal{S}(\R^{n-1})$ is given by 
$$
  S^\eps_{x_1,z_1,\tilde z_1}g := \bar A^\eps (x_1-\tilde z_1,\cdot)  \ast A^\eps(x_1- z_1,\cdot) \ast \bar g.
$$
The $(L^2,L^2)$-bound for $S^\eps_{x_1,z_1,\tilde z_1}$ results from
\begin{align} \label{eq:L2estimate}
  \begin{aligned}
  \|S^\eps_{x_1,z_1,\tilde z_1}g\|_{L^2(\R^{n-1})}
  &= \|\mathcal F_{n-1}(\bar A^\eps (x_1-\tilde z_1,\cdot)  \ast A^\eps(x_1- z_1,\cdot)) \cdot 
  \bar{\hat g}\|_{L^2(\R^{n-1})} \\
  &\leq \|\mathcal F_{n-1}\left(\bar A^\eps(x_1-\tilde z_1,\cdot) \ast A^\eps(x_1-
  z_1,\cdot)\right)\|_{L^\infty(\R^{n-1})} \|g\|_{L^2(\R^{n-1})}   \\
  &\leq \sup_{t\in\R} \|\mathcal F_{n-1}(\bar A^\eps(t,\cdot))\|_{L^\infty(\R^{n-1})}^2
   \|g\|_{L^2(\R^{n-1})}  \\
  &\leq \sup_{t\in\R,\xi'\in\R^{n-1}}|\mathcal{F}_1^{-1} (\hat A^\eps(\cdot,\xi'))(t)|^2 \cdot
  \|g\|_{L^2(\R^{n-1})} \\
  &= \sup_{t\in\R,\xi'\in\R^{n-1}} |\Psi_t^\eps(\xi')|^2 \cdot \|g\|_{L^2(\R^{n-1})} \\
  &\stackrel{\eqref{eq:esimatePsiteps3}}{\leq}  C \|g\|_{L^2(\R^{n-1})}.
  \end{aligned}
\end{align}
The $(L^1,L^\infty)$-bound is obtained via the method of stationary phase. 
\begin{align*}
  \|S^\eps_{x_1,z_1,\tilde z_1}g\|_\infty
  &\leq \left\| \bar A^\eps (x_1-\tilde z_1,\cdot)  \ast A^\eps(x_1- z_1,\cdot)
  \right\|_\infty \|g\|_1 \\
  &= \left\|\mathcal F_{n-1}^{-1} \left(  \ov{\mathcal{F}_1^{-1}(\hat A^\eps(\cdot,\xi'))(x_1-\tilde
  z_1)} \cdot  \mathcal{F}_1^{-1} (\hat A^\eps(\cdot,\xi'))(x_1-z_1) \right) \right\|_\infty \|g\|_1\\
  &\stackrel{\eqref{eq:esimatePsiteps3}}{=} \left\|\mathcal F_{n-1}^{-1} \left(
  e^{-i\sqrt{1-|\xi'|^2}(x_1-\tilde z_1)} \ov{\Psi^\eps_{x_1-\tilde z_1}(\xi')}\cdot
   e^{i\sqrt{1-|\xi'|^2}(x_1-z_1)}  \Psi^\eps_{x_1-z_1}(\xi') \right) \right\|_\infty \|g\|_1\\
  &= \left\|\mathcal F_{n-1}^{-1} \left( e^{i\sqrt{1-|\xi'|^2}(\tilde z_1- z_1)} 
   \ov{\Psi^\eps_{x_1-\tilde z_1}(\xi')} \Psi^\eps_{x_1-z_1}(\xi') \right) \right\|_\infty \|g\|_1\\
   &= \frac{1}{(2\pi)^{\frac{n-1}{2}}} \sup_{y'\in\R^{n-1}} \left| \int_{\R^{n-1}}
   e^{i(\skp{y'}{\xi'}+\sqrt{1-|\xi'|^2} (z_1-\tilde z_1) )} \ov{\Psi^\eps_{x_1-\tilde z_1}(\xi')}
   \Psi^\eps_{x_1-z_1}(\xi')    \,d\xi'   \right| \|g\|_1.
\end{align*}
In the last integral the smooth phase function $\Phi(\xi'):=\skp{y'}{\xi'}+\sqrt{1-|\xi'|^2}(z_1-\tilde z_1)$ is 
stationary precisely at $\xi' = \frac{\sign(z_1-\tilde z_1)}{\sqrt{|y'|^2+(z_1-\tilde
z_1)^2}}y'$ and the Hessian  in that point
$$ 
  D^2\Phi(\xi') =   \frac{\tilde z_1-z_1}{\sqrt{1-|\xi'|^2}} \left(\Id+\frac{1}{1-|\xi'|^2}\xi'(\xi')^T\right)
  \in\R^{(n-1)\times (n-1)}
$$
possesses the eigenvalues $1,\ldots,1,\frac{1}{1-|\xi'|^2}$, which are all
uniformly bounded away from zero and infinity on the support of $\xi'\mapsto \ov{\Psi^\eps_{x_1-\tilde
z_1}(\xi')} \Psi^\eps_{x_1-z_1}(\xi')$. Moreover, by~\eqref{eq:esimatePsiteps3} all
derivatives of this function are square integrable with $L^2$-norms that are uniformly bounded with respect to
$x_1,\tilde z_1,z_1,\eps$. Hence, the Morse Lemma and the method of stationary phase yield
\begin{equation}  \label{eq:L1estimate}
  \|S^\eps_{x_1,z_1,\tilde z_1}g\|_{L^\infty(\R^{n-1})}
  \leq C(1+|z_1 - \tilde z_1|)^{\frac{1-n}{2}} \|g\|_{L^1(\R^{n-1})}.
\end{equation}
Interpolating the $(L^2,L^2)$-estimate~\eqref{eq:L2estimate} and the
$(L^1,L^\infty)$-estimate~\eqref{eq:L1estimate} we get from the Riesz-Thorin Theorem
\begin{align} \label{eq:ruizvega}
    \|S_{x_1,z_1,\tilde z_1}g\|_{L^{p^\prime}(\R^{n-1})}
    \leq C(1+|z_1 -  \tilde z_1|)^{(1-n)(\frac{1}{p}-\frac{1}{2})}\|g\|_{L^p(\R^{n-1})}
    \qquad\text{for all }p\in [1,2]. 
\end{align}
With this estimate we are finally ready to conclude.

\medskip

So assume $\frac{1}{q}-\frac{1}{2}\geq \frac{1}{n+1}$. Then
$(1+|\cdot|)^{(1-n)(\frac{1}{q}-\frac{1}{2})}$ lies in $L^{\frac{q}{2(q-1)},w}(\R^n)$ so that Young's
inequality for weak Lebesgue spaces implies
\begin{align*}
   &\frac{1}{R}\int_{B_R} |v^\eps (x)|^2+ |\nabla v^\eps (x)|^2 \,dx \\ 
   &\stackrel{\eqref{eq:est1}}{\leq}
    \sup_{x_1\in\R} \|(A^\eps\ast f)(x_1,\cdot) \|_{L^2(\R^{n-1})}^2 \\
    &= \sup_{x_1\in\R} \int_{\R} \int_\R \Big( \int_{\R^{n-1}} f(z_1,y') S^\eps_{x_1,z_1,\tilde z_1}(f(\tilde
    z_1,\cdot))(y') \,dy' \Big) \,dz_1 \,d\tilde z_1 \\
   &\leq \sup_{x_1\in\R}  \int_\R \int_\R \|f(z_1,\cdot)\|_{L^q(\R^{n-1})}  \| S^\eps_{x_1,z_1,\tilde
   z_1}(f(\tilde z_1,\cdot))\|_{L^{q'}(\R^{n-1})} \,dz_1 \,d\tilde z_1 \\
   &\stackrel{\eqref{eq:ruizvega}}{\leq}  C\int_\R \int_\R 
    (1+|z_1-\tilde z_1|)^{(1-n)(\frac{1}{p}-\frac{1}{2})}\|f(z_1,\cdot)\|_{L^q(\R^{n-1})}   
   \|f(\tilde z_1,\cdot)\|_{L^q(\R^{n-1})} \,dz_1 \,d\tilde z_1 \\
   &\leq C\|f\|_q^2.   
\end{align*}
This finally shows~\eqref{eq:estimate_veps} and the proof is finished. \qed

\section*{Acknowledgements}

The author gratefully acknowledges financial support by the Deutsche Forschungsgemeinschaft (DFG) through CRC
1173 ''Wave phenomena: analysis and numerics''.

\bibliographystyle{plain}
\bibliography{doc}
  
\end{document}